\documentclass[a4paper,12pt]{article}

\usepackage[a4paper,top=2cm,bottom=2cm,left=2cm,right=2cm,bindingoffset=5mm]{geometry}
\usepackage[english,italian]{babel}
\usepackage[T1]{fontenc}
\usepackage{amsfonts}
\usepackage{amsmath}
\usepackage{amssymb}
\usepackage{amsthm}
\usepackage{fancyhdr}
\usepackage{indentfirst}
\usepackage{latexsym}
\usepackage{graphicx}
\usepackage{graphics}
\usepackage{wrapfig}
\usepackage{eufrak}
\usepackage{float}
\usepackage{xy}
\input xy
\xyoption{all}
\usepackage{mathpazo}
\usepackage[scaled=.95]{helvet}
\usepackage{courier}
\usepackage{subfig}
\usepackage{braket}
\usepackage{hyperref}
\hypersetup{colorlinks=true,linkcolor=blue}

\newtheorem{teo}{Theorem}[section]

\newtheorem{prop}[teo]{Proposition}
\newtheorem{corol}[teo]{Corollary}

\theoremstyle{definition}
\newtheorem{defi}[teo]{Definition}

\theoremstyle{remark}
\newtheorem{oss}[teo]{Remark}

\newcommand{\R}{\mathbb{R}}

\newcommand{\elle}{\mathcal{L}}

\DeclareMathOperator{\spann}{span}

\DeclareMathOperator{\ext}{Ext}

{\left\lbrace\begin{array}{@{}l@{}}}%
{\end{array}\right.}

\begin{document}
\selectlanguage{english}%

\title{\bf{Nonlocal Venttsel' diffusion in fractal-type domains: regularity results and numerical approximation}}
\author{ Massimo Cefalo\footnote{Dipartimento di Ingegneria Informatica, Automatica e Gestionale,
	Universit\`{a} degli Studi di Roma "La Sapienza", Via Ariosto 25, 00185 Roma, Italy. E-mail: cefalo.m@gmail.com}, Simone Creo\thanks{Dipartimento di Scienze di Base e Applicate per l'Ingegneria,  Sapienza, Universit\`{a}  di Roma
    Via A. Scarpa 16, 00161 Roma, Italy.
   E-mail: simone.creo@sbai.uniroma1.it, mariarosaria.lancia@sbai.uniroma1.it},
Maria Rosaria Lancia$^\dag$, Paola Vernole\footnote{ Dipartimento di Matematica, Universit\`{a} degli Studi di Roma "La Sapienza",
    P.zale Aldo Moro 2,   00185 Roma, Italy. E-mail vernole@mat.uniroma1.it}
		}
\date{}
\maketitle

\begin{abstract}

\noindent We study a nonlocal Venttsel' problem in a non-convex bounded domain with a Koch-type boundary. Regularity results of the strict solution are proved in weighted Sobolev spaces. The numerical approximation of the problem is carried out and optimal a priori error estimates are obtained.
\end{abstract}



 \pagestyle{myheadings} \thispagestyle{plain} \markboth{M. CEFALO, S. CREO, M. R. LANCIA AND P. VERNOLE}{Nonlocal Venttsel' diffusion in fractal-type domains}

\section*{Introduction}

\noindent In this paper we study a nonlocal Venttsel' problem in a non-convex bounded domain with a Koch-type boundary.\\
Recently there has been an increasing interest towards the study of Venttsel' problems in fractal domains due to the different framework in
which they appear, e.g. engineering problems of idraulic fracturing, water wave theory as well as  models of heat transfer    (see 	\cite{canmey}, \cite{sanpal}, \cite{Kor} and \cite{Shim}).\\
In the framework of heat propagation, a Venttsel' problem models the heat flow across highly conductive thin boundaries (see \cite{Fa-La-Le-Ma}). From the point of view of applications it is important to consider models in which the surface effects are enhanced with respect to the surrounding volume. Fractal boundaries or interfaces can be a useful tool to describe such situation. It has to be pointed out that many physical and industrial processes lead to the formation of irregular surfaces or occur across them which can be conveniently modeled as prefractal interfaces.\\
The literature on linear and quasi-linear Venttsel' problems in both regular and irregular domains is huge (see e.g. \cite{Ar-Me-Pa-Ro}, \cite{Fa-G-G-Ro}, \cite{velezjfa14}, \cite{WAR12-2}, \cite{lanver2} and the references listed in). In all these papers it is addressed the study of existence and uniqueness properties of the solution by a semigroup approach. Only recently in \cite{lanver2} and in \cite{nostro} the asymptotic behavior of the solution of local linear and quasi-linear prefractal Venttsel' problems to the limit fractal ones has been investigated (respectively).\\
A Venttsel' problem is described mathematically by a heat equation in the bulk coupled with a heat equation on the boundary. Due to this
unusual boundary condition, Venttsel' problems are also known as BVPs with dynamical boundary conditions since the time derivative appears also in the boundary condition.\\
Recently there has been a growing attention on the study of nonlocal Venttsel' problems both in regular and irregular domains in $\R^n$ (see \cite{WAR12}, \cite{AVS-WAR}, \cite{velezjfa15}, \cite{LVSV} and the references listed in). The presence of the nonlocal term in the boundary condition, in the framework of heat flow, accounts for a non-constant conductivity $K(x,y)$ on the boundary which scales according to a certain law:

$$ \frac{k^{-1}}{|x-y|^{n+2s}}\leq K(x,y)\leq \frac{k}{|x-y|^{n+2s}}$$
where $s\in [0,1]$ and $k$ is a positive constant (see \cite{bass-Levin} for more details and for a probabilistic interpretation of the associated process).\\
It has to be pointed out that a nonlocal term appears already in the original paper of Venttsel' \cite{vent59}. In any case a nonlocal term is important in all those diffusion models in which one wants to emphasize the interaction between the boundary and the bulk such as e.g. in the diffusion of sprays in the lungs.\\
In this paper we consider a parabolic nonlocal Venttsel' problem $(P)$ in a two-dimensional domain $\Omega$ with a Koch-type prefractal boundary and its numerical approximation, see \eqref{astratto}.\\
As far as we know, this is the first example of a numerical approximation of a nonlocal Venttsel' boundary value problem in a prefractal domain. The numerical approximation of second order transmission problems across prefractal interfaces have been considered in \cite{lancef} and in \cite{haodong}. This type of problems exhibits on the interface a local Venttsel'-type boundary condition. In all the mentioned cases, once the existence and uniqueness of the strict solution is proved, the regularity (of the solution) plays a key role in the error estimates as well as a suitable refined mesh near the singular vertices. Indeed, since the prefractal domains are not convex due to the presence of the reentrant angles, the regularity of the solution is deteriorated.\\
In the nonlocal case, in order to prove key regularity results for the strict solution of problem $(P)$, a completely different approach (with respect to the case of local type Venttsel' boundary conditions) must be used due to the presence of the nonlocal term, which can be regarded as a sort of "regional" fractional Laplacian of order $\frac{1}{2}$.\\
Namely, we prove that the solution belongs to a suitable weighted Sobolev space. This allows us to use a convenient mesh algorithm developed in \cite{cefalolancia} compliant to suitable conditions introduced by Grisvard in \cite{grisvard} which allow us to achieve an optimal rate of convergence.\\
The numerical approximation of problem $(P)$ is carried out in two steps. We first triangulate the domain $\Omega$ with our mesh algorithm. We construct the semi-discrete problem by discretizing in space. Secondly the fully discrete problem is obtained by applying the $\theta$-method to the time-dependent variables.\\
We obtain stability and convergence results as in the classical case, where the solution has $H^2$ regularity. It is worthwhile noticing that our results can be straightforward extended to the more general case of domains with boundaries of Koch type fractal mixtures as in \cite{haodong}.\\
We achieve also some preliminary numerical results. We study the heat flow across a prefractal boundary where the nonlocal term is active only on a portion of its. As shown in Figures \ref{fig:streamlines} and \ref{fig:streamlines3D}, the nonlocal term is responsible of a larger heat flux in the part of the boundary where it is active. From the point of view of the applications, this fact turns out to be important to drain or increase the heat in \textsl{a priori} fixed areas.

\noindent The plan of the paper is the following. In Section \ref{sec1} we recall the preliminaries on the Koch curve and the main functional spaces. In Section \ref{sec2} we state the main properties of the energy functional. In Section \ref{secapriori} we prove a priori estimates in weighted Sobolev spaces. In Section \ref{sec3} we consider the abstract Cauchy problem and we prove existence and uniqueness results for the strict solution of the nonlocal Venttsel' problem in the prefractal domain $\Omega$. In Section \ref{sec4} we prove regularity results in fractional Sobolev spaces. In Section \ref{sec5} we prove a priori error estimates for the semi-discrete and fully discrete problem. Finally, in Section \ref{sec6} we present some numerical results and conclusions.

\section{Preliminaries}\label{sec1}
\setcounter{equation}{0}

\noindent In the paper  we denote by $x=(x_1,x_2)$ points in $\R^2$. Let $A_1$, $A_3$ and $A_5$ be the vertices of a regular triangle with unit side length, i.e. $|A_1-A_3|=|A_1-A_5|=|A_3-A_5|=1$. We define a family $\Psi^{1}$ of four suitable contractions $\psi_{1}^{(1)},...,\psi_{4}^{(1)}$, with respect to the same ratio $\frac{1}{3}$ (see \cite{freiberg}). Let $V_0^{(1)}:=\lbrace A_1,A_3 \rbrace$, $\psi_{i_1\dots i_n}:=\psi_{i_1}\circ\dots\circ\psi_{i_n}$, $V_{i_1\dots i_n}^{(1)}:=\psi_{i_1\dots i_n}^{(1)}(V_0^{(1)})$ and
	\begin{center}
	$V_n^{(1)}:=\bigcup\limits_{i_1\dots i_n=1}^4 V_{i_1\dots i_n}^{(1)}$.
	\end{center}
Now let $K_1^{(0)}$ denote the unit segment whose endpoints are $A_1$ and $A_3$ and let $K^{(1)}_{i_1,\dots i_n}:=\psi_{i_1\dots i_n} (K_1^{(0)})$. For $n>0$, we denote
\begin{center}
$\displaystyle F^n_{(1)}=\left\{\psi_{i_1\dots i_n} (K_1^{(0)})\,,\,i_1,\dots,i_n=1,\dots,4\right\}$.
\end{center}
Now we set $\displaystyle K_1^{(1)}=\bigcup_{i=1}^4 \psi_i(K_1^{(0)})$, $\displaystyle K_1^{(n+1)}=\bigcup_{M\in F^n_{(1)}}\bigcup_{i=1}^4 \psi_i(M)$, where $M$ denotes a segment of the $(n+1)$-th generation, $K_1^{(n+1)}$ the polygonal curve and $V_{n+1}^{(1)}$ the set of its vertices.\\
We can repeat this argument on the unit segments $K_2^{(0)}$ having as endpoints $A_3$ and $A_5$ and $K_3^{(0)}$ having as endpoints $A_1$ and $A_5$. In the following we denote by
\begin{equation}\label{eq:3bitris}
K_{n+1}=\displaystyle\bigcup_{i=1}^3 K_i^{(n+1)}
\end{equation}
the closed polygonal curve approximating the Koch snowflake at the $(n+1)$-th step. For the properties of the Koch snowflake we refer to \cite{falconer}.
\begin{figure}
	\centering
		\includegraphics[width=0.60\textwidth]{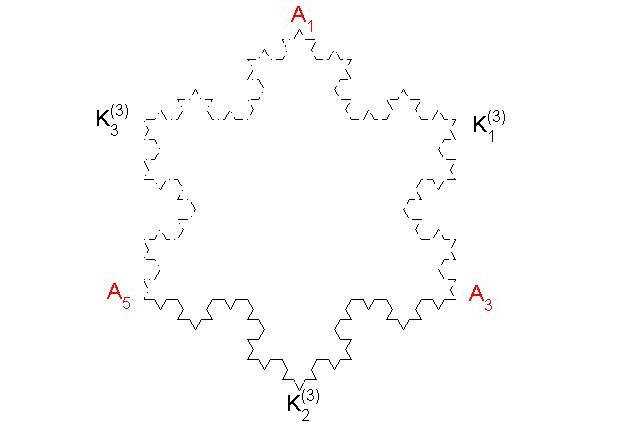}
	\caption{The domain $\Omega_n$ for $n=3$.}
	\label{dominio}
\end{figure}
\noindent By $\Omega_n$ we denote the open bounded non-convex domain in $\mathbb{R}^2$ with polygonal boundary $\partial\Omega_n=K_n$ with vertices $P_j$ for $j=1,\dots,3\mathcal{N}$, where $\mathcal{N}=4^n$. Since $n$ in this paper will be fixed, from now on we will omit the subscript $n$ in $\Omega_n$.\\
For each fixed $P_j$, we denote by $\eta_j$ the interior angle of $\Omega$ at $P_j$. Let $R=\{1\leq j\leq 3\mathcal{N} \,:\,\eta_j>\pi\}$. The set $\{P_j\}_{j\in R}$ is the subset of vertices whose angles are "reentrant".

\medskip
\noindent In the following we will denote the set of continuous functions on $\Omega$ by $C^0 (\Omega)$, the set of infinitely differentiable functions with compact support in $\Omega$ by $C^{\infty}_0 (\Omega)$ and by $L^2(\cdot)$ the usual Lebesgue space of square summable functions.\\
We define the Sobolev spaces on $K_n$. By $H^s(K_n)$, for $0<s\leq 1$, we denote the Sobolev space on $K_n$ defined by local Lipschitz charts as in~\cite{necas}. For $s\geq 1$, we define the Sobolev space $H^s(K_n)$ by using the characterization given by Brezzi-Gilardi in~\cite{bregil}:
\begin{center}
$H^s (K_n)=\{v\in C^0 (K_n)\,|\, v|_{\overset{\circ}{M}}\in H^s (\overset{\circ}{M})\}$,
\end{center}
where $M$ denotes a side of $K_n$ and $\overset{\circ}{M}$ denotes the corresponding open segment. In the following we denote by $dx$ the bidimensional Lebesgue measure and by $ds$ the arc length on $K_n$.\\
We define the trace of $u$ on $K_n$ in the following way.
\begin{defi}
Let $D\subset\mathbb{R}^2$ be an open set and $f\in H^s (D)$. We define the trace operator $\gamma_0$ as
\begin{center}
$\displaystyle\gamma_0 f(x):=\lim_{r\to 0} \frac{1}{|B_r (x)\cap D|}\,\int_{B_r (x)\cap D} f(y)\,dy$,
\end{center}
where $B_r (x)$ denotes the Euclidean ball in $\mathbb{R}^2$.
\end{defi}
\noindent It is known that the above limit exists at \emph{quasi every} $x\in\overline{D}$ with respect to the $(s,2)$-capacity (see~\cite{adhei}).

\noindent In the following we will make use of this result (see Theorem 2.4 in~\cite{bregil}).
\begin{prop}\label{traccia}
Let $s>1/2$. Then $H^{s-\frac{1}{2}} (K_n)$ is the trace space of $H^s (\Omega)$, i.e.:
\begin{enumerate}
\item $\gamma_0$ is a linear and continuous operator from $H^s (\Omega)$ in $H^{s-\frac{1}{2}} (K_n)$;
\item there exists a linear and continuous operator $\ext$ from $H^{s-\frac{1}{2}} (K_n)$ in $H^s (\Omega)$ such that $\gamma_0\circ\ext$ is the identity operator of $H^{s-\frac{1}{2}} (K_n)$.
\end{enumerate}
\end{prop}
\noindent We will use the same symbol to denote $u$ and its trace $\gamma_0 u$. The interpretation will be left to the context.\\
We now recall the Friedrichs inequality, see \cite[page 24]{mazya} for more details.
\begin{prop} Let $u\in H^1(\Omega)$. There exists a positive constant $C_p$ such that
\begin{equation}\label{poincare}
\|u\|^2_{L^2(\Omega)}\leq C_p \left(\|\nabla u\|^2_{L^2(\Omega)}+\|u\|^2_{L^2(K_n)}\right).
\end{equation}
\end{prop}

\noindent Let $r$ be the distance from the set of vertices. For $\gamma\in\R$, we denote by $H^2_\gamma (\Omega)$ the weighted Sobolev space of functions such that the norm
\begin{center}
$\displaystyle\|u\|_{H^2_\gamma (\Omega)}=\left(\sum_{|\alpha|\leq 2} \int_{\Omega} r^{2(\gamma-2+|\alpha|)} |D^{\alpha} u(x)|^2\,dx\right)^{\frac{1}{2}}$
\end{center}
is finite, and by the space $H^{\frac{3}{2}}_\gamma (K_n)$ the trace space of $H^2_\gamma (\Omega)$ equipped with the norm
\begin{center}
$\displaystyle\|u\|_{H^{\frac{3}{2}}_\gamma (K_n)}=\inf_{v=u\,\text{on}\,K_n}\,\|v\|_{H^2_\gamma (\Omega)}$.
\end{center}
For the details see (2.17), Chapter 2 in \cite{nazplam}.\\ 
We denote by $L^2(\Omega,m)$ the Lebesgue space with respect to the measure $m$ with
\begin{center}
$dm=dx+ds$.
\end{center}
We define, for $\sigma\in\R$, the composite space
$$V^2_\sigma(\Omega,K_n):=\{u\in H^1 (\Omega)\,:\, r^\sigma D^2 u\in L^2(\Omega),\,\gamma_0 u\in H^2(K_n)\}.$$

\section{The energy functional}\label{sec2}
\setcounter{equation}{0}

\noindent Throughout the paper $C$ will denote possibly different constants.\\
Let $b$ be a positive continuous function on $\overline{\Omega}$. We set $\theta_2\colon H^{\frac{1}{2}} (K_n)\to H^{-\frac{1}{2}} (K_n)$ as follows: for every $u,v\in H^{\frac{1}{2}} (K_n)$
\begin{center}
$\displaystyle\langle\theta_2 (u),v\rangle_{H^{-\frac{1}{2}} (K_n), H^{\frac{1}{2}} (K_n)}=\iint_{K_n\times K_n}\frac{(u(x)-u(y))(v(x)-v(y))}{|x-y|^2}\,ds(x)\,ds(y)$.
\end{center}
From now on $\langle\cdot,\cdot\rangle$ will denote the duality pairing between $H^{-\frac{1}{2}} (K_n)$ and $H^{\frac{1}{2}} (K_n)$.\\
We define now the energy form $E$ as
\begin{equation}\label{defforma}
E[u]=E_{\Omega} [u]+E_{K_n} [u]+\int_{K_n} b\,|u|^2\,ds+\langle\theta_2 (u),u\rangle
\end{equation}
with domain
\begin{center}
$V(\Omega, K_n)=\{u\in H^1 (\Omega)\,:\, \gamma_0 u\in H^1 (K_n)\}$,
\end{center}
where $$E_{\Omega} [u]=\int_{\Omega}|\nabla u|^2\,dx$$ and $$E_{K_n}[u]=\int_{K_n}|\nabla_s u|^2\,ds.$$ Here $\nabla_s$ denotes the tangential derivative on $K_n$.\\
$V(\Omega,K_n)$ is a Hilbert space equipped with the norm
\begin{center}
$\displaystyle\|u\|_{V(\Omega,K_n)}=\left(\int_{\Omega}|\nabla u|^2\,dx+\int_{K_n}|\nabla_s u|^2\,ds+\|u\|^2_{L^2(\Omega,m)}\right)^{\frac{1}{2}}$.
\end{center}
We point out that the space $V(\Omega, K_n)$ is non-trivial.\\
In order to prove the coercivity of the energy form $E$, we introduce an equivalent norm on $V(\Omega,K_n)$, which is defined as
\begin{equation}\label{normeq}
|||u|||_{V(\Omega,K_n)}:=\left(\|\nabla u\|^2_{L^2(\Omega)}+\Phi(u)\right)^{\frac{1}{2}},
\end{equation}
where $\Phi(u):=\|u\|^2_{H^1(K_n)}+\langle\theta_2(u),u\rangle$.
\begin{prop} The norms $\|\cdot\|_{V(\Omega,K_n)}$ and $|||\cdot|||_{V(\Omega,K_n)}$ are equivalent.
\end{prop}
\begin{proof} We note that $\|u\|^2_{V(\Omega, K_n)}\leq C_2 |||u|||^2_{V(\Omega,K_n)}$ thanks to \eqref{poincare}. To prove that $\|u\|^2_{V(\Omega,K_n)}\geq C_1 |||u|||^2_{V(\Omega,K_n)}$ we note that
\begin{center}
$\displaystyle\langle\theta_2 (u),u\rangle\leq\|u\|^2_{H^{\frac{1}{2}} (K_n)}\leq\,C\|u\|^2_{H^1(K_n)}$.
\end{center}
\end{proof}
\noindent We now prove some properties of the form $E$.
\begin{prop}\label{coer} The form $E[u]$ defined in \eqref{defforma} is continuous and coercive on $V(\Omega,K_n)$.
\end{prop}
\begin{proof} We start by proving the continuity of $E$ on $V(\Omega,K_n)$. Since $b$ is continuous on $\overline{\Omega}$, we have
\begin{center}
$\displaystyle E[u]\leq\|\nabla u\|^2_{L^2(\Omega)}+\|\nabla_s u\|^2_{L^2(K_n)}+\max_{K_n}\,(b)\,\|u\|^2_{L^2(K_n)}+\langle\theta_2 (u),u\rangle\leq\|u\|^2_{H^1(\Omega)}+c_1\,\|u\|^2_{H^1(K_n)}+\langle\theta_2 (u),u\rangle\leq\|u\|^2_{H^1(\Omega)}+c_2\|u\|^2_{H^1(K_n)}\leq\max\{1,c_2\}\,\|u\|^2_{V(\Omega, K_n)}$.
\end{center}
We prove the coercivity. By using again the continuity of $b$, we have
\begin{center}
$\displaystyle E[u]\geq\|\nabla u\|^2_{L^2(\Omega)}+\|\nabla u\|^2_{L^2(K_n)}+\min_{K_n}\,(b)\,\|u\|^2_{L^2(K_n)}+\langle\theta_2 (u),u\rangle\geq\|\nabla u\|^2_{L^2(\Omega)}+\min\{1,\min_{K_n}\,(b)\}\|u\|^2_{H^1(K_n)}+\iint_{K_n\times K_n}\frac{|u(x)-u(y)|^2}{|x-y|^2}\,ds(x)\,ds(y)\geq\bar{C}\|u\|^2_{V(\Omega,K_n)}$,
\end{center}
where $\bar{C}$ depends on $b$ and $C_1$.
\end{proof}

\begin{prop} The energy form $E[u]$ is closed in $L^2(\Omega,m)$, i.e. for every Cauchy sequence $\{u_k\}\subseteq V(\Omega,K_n)$ there exists $u\in V(\Omega,K_n)$ such that
\begin{center}
$\displaystyle E[u_k-u]+\|u_k-u\|_{L^2(\Omega,m)}\to 0\quad$ for $k\to +\infty$.
\end{center}
\end{prop}
\begin{proof} Let $\{u_k\}$ be a Cauchy sequence in $V(\Omega,K_n)$, i.e. a sequence such that
\begin{center}
$E[u_k-u_j]+\|u_k-u_j\|_{L^2(\Omega,m)}\to 0\quad$ for $k,j\to +\infty$.
\end{center}
We observe that in particular $\{u_k\}$ is a Cauchy sequence in $L^2(\Omega,m)$ and, since $L^2(\Omega,m)$ is a Banach space, there exists an element $u\in L^2(\Omega,m)$ such that
\begin{center}
$\|u_k-u\|_{L^2(\Omega,m)}\xrightarrow[k\to +\infty]{} 0$.
\end{center}
We have to prove that
\begin{center}
$\displaystyle E[u_k-u]\xrightarrow[k\to +\infty]{} 0$.
\end{center}
We note that from $E[u_k-u_j]\to 0$ when $k,j\to +\infty$, it follows that each term in \eqref{defforma} vanishes (because they are all non negative terms).\\
Since $\displaystyle\int_{\Omega} |\nabla (u_k-u_j)|^2\,dx\to 0$, it follows that $\{\nabla u_k\}$ is a Cauchy sequence in $L^2 (\Omega)$, and the same holds for the terms in $L^2(K_n)$. Then $\{\nabla u_k\}$ is a Cauchy sequence in $L^2 (\Omega,m)$, hence there exists an element $w\in L^2(\Omega,m)$ such that $\nabla u\to w$ in $L^2 (\Omega,m)$. From Remark 4, Chapter 9 in~\cite{brezis}, we know that $w=\nabla u$ a.e., so we have that $u\in V(\Omega,K_n)$.\\
It is trivial that $\displaystyle\int_{K_n} b|u_k-u|^2\,ds\to 0$ because $b$ is a continuous function on $\bar{\Omega}$. It remains to study the term $\theta_2$:
\begin{center}
$\displaystyle\langle\theta_2 (u_k-u),u_k-u\rangle_{H^{-\frac{1}{2}} (K_n), H^{\frac{1}{2}} (K_n)}=\iint_{K_n\times K_n}\frac{|u_k(x)-u(x)-(u_k(y)-u(y))|^2}{|x-y|^2}\,ds(x)\,ds(y)\leq\|u_k-u\|^2_{H^{\frac{1}{2}} (K_n)}\leq\,C\|u_k-u\|^2_{H^1(K_n)}$
\end{center}
and the last term tends to 0 when $k\to +\infty$ because we know that $u_k\to u$ in $V(\Omega,K_n)$.
\end{proof}

\begin{teo}\label{dirform} The energy form $E[u]$ with its domain $V(\Omega,K_n)$ is a Dirichlet form on $L^2(\Omega,m)$.
\end{teo}
\begin{proof} We have to prove that $E[u]$ is markovian. Since we know that $E[u]$ is closed, we can prove a sufficient condition for having markovianity, i.e. $\forall\,u\in V(\Omega,K_n)$
\begin{center}
$v:=(u\vee 0)\wedge 1\in V(\Omega, K_n)$ and $E[v]\leq E [u]$,
\end{center}
where $\displaystyle (u\vee v)(x)=\max\,\{u(x), v(x)\}$ and $\displaystyle (u\wedge v)(x)=\min\,\{u(x), v(x)\}$.\\
Let us consider the map $T\colon\mathbb{R}\to\mathbb{R}$ defined as $T(r)=((r\vee 0)\wedge 1)$, then we set $v(x):=T(u(x))$. Now we approximate $T$ with functions $T_{\varepsilon}\in C^1 (\mathbb{R})$ such that
\begin{center}
$|T_{\varepsilon} (r)-T(r)|<\varepsilon$ and $\displaystyle \left|{\frac{dT_{\varepsilon}}{dr}}\right|\leq 1$.
\end{center}
Since $T_{\varepsilon}\in C^1 (\mathbb{R})$ and $u\in V(\Omega, K_n)$, it follows that $T_{\varepsilon} (u(x))\in V(\Omega, K_n)$, then $T(u(x))=v(x)\in V(\Omega, K_n)$. Now\newline\newline
$\begin{array}{lll}
 \hspace{1.5 cm}\displaystyle\frac{dT_{\varepsilon} (u(x_1(s),x_2(s)))}{ds}=\frac{\partial T_{\varepsilon}}{\partial u}\,\frac{\partial u}{\partial x_1}\,\frac{dx_1(s)}{ds}+\frac{\partial T_{\varepsilon}}{\partial u}\,\frac{\partial u}{\partial x_2}\,\frac{dx_2(s)}{ds}\\[5mm]
 \hspace{3.6 cm}=\displaystyle\frac{\partial T_{\varepsilon}}{\partial u}\,(\nabla u, z(s))=\frac{\partial T_{\varepsilon}}{\partial u}\,\nabla_s u,
\end{array}$\\\newline
where $z(s)=(x'_1(s),x'_2(s))$. Using the properties of $T_{\varepsilon}$, it follows that
\begin{center}
$\displaystyle\left|{\frac{dT_{\varepsilon} (u(x_1(s),x_2(s)))}{ds}}\right|^2\leq\left|{\frac{\partial T_{\varepsilon}}{\partial u}}\right|^2\,|\nabla_s u(x_1(s),x_2(s))|^2\leq |\nabla_s u(x_1(s),x_2(s))|^2$.
\end{center}
Hence we have that
\begin{center}
$\displaystyle E_{K_n} [v]=E_{K_n} [T(u)]\leq\int_{K_n} |\nabla_s u(x(s))|^2\,ds=E_{K_n} [u]$.
\end{center}
We can repeat the same argument on $\displaystyle E_{\Omega} [u]$ to prove that $E_{\Omega}[v]\leq E_{\Omega}[u]$. It can be easily seen that this holds also for the other terms in \eqref{defforma} (see e.g. \cite{farkas}). Hence $E[u]$ is markovian, then $E[u]$ with its domain $V(\Omega,K_n)$ is a Dirichlet form on $L^2(\Omega,m)$.
\end{proof}

\noindent For the main properties of Dirichlet forms, see~\cite{fukush}.\\
\noindent Now we define the bilinear form associated to the energy form $E[u]$ as follows: for every $u,v\in V(\Omega,K_n)$
\begin{equation}
E(u,v)=\int_{\Omega} \nabla u\,\nabla v\,dx+\int_{K_n} \nabla_s u\,\nabla_s v\,ds+\int_{K_n} b\,u\,v\,ds+\langle\theta_2 (u),v\rangle.
\end{equation}
\begin{prop} For all $u,v\in V(\Omega,K_n)$, $E(u,v)$ is a closed symmetric bilinear form on $L^2(\Omega,m)$. Then there exists a unique self-adjoint non-positive operator $A$ on $L^2(\Omega,m)$ such that
\begin{equation}\label{corr}
E(u,v)=(-Au,v)_{L^2(\Omega,m)}\quad\forall\,u\in D(A),\forall\,v\in V(\Omega,K_n),
\end{equation}
where $D(A)\subset V(\Omega,K_n)$ is the domain of $A$ and it is dense in $L^2(\Omega,m)$.
\end{prop}

\noindent For the proof see~\cite{kato}.\\
In Theorem \ref{dirform} we proved that $(E_{K_n}, H^1_0 (K_n))$ is a closed bilinear form on $L^2 (K_n)$. Then, there exists a unique self-adjoint, non positive operator $\Delta_{K_n}$ on $L^2 (K_n)$ (with domain $D(\Delta_{K_n})$ dense in $L^2 (K_n)$) such that
\begin{center}
$\displaystyle E_{K_n} (u,v)=-\int_{K_n} (\Delta_{K_n}\,u)\,v\,ds,\quad u\in D(\Delta_{K_n}),v\in H^1_0(K_n)$
\end{center}
(see Chap. 6, Theorem 2.1 in~\cite{kato}). Let now $H^{-1}(K_n)$ be the dual space of $H^1_0(K_n)$. We can also introduce the Laplace operator on $K_n$ as a variational operator $$\Delta_{K_n}\colon H^1_0(K_n)\to H^{-1}(K_n)$$ by
\begin{equation}\label{var}
E_{K_n} (z,w)=-\langle\Delta_{K_n}\,z,w\rangle_{H^{-1}(K_n), H^1_0(K_n)}
\end{equation}
for $z\in H^1_0(K_n)$ and for all $w\in H^1_0(K_n)$. We will use the same symbol $\Delta_{K_n}$ to define the Laplace operator both as a self-adjoint operator and as a variational operator. It will be clear from the context to which case we refer.
\begin{oss} As it will be clear in \eqref{pbforte2}, $\Delta_{K_n}$ will be the piecewise tangential Laplacian with domain $D(\Delta_{K_n})=H^2(K_n)$
\end{oss}
\begin{teo} The self-adjoint non positive operator $A$ associated to the Dirichlet form $E[u]$ is the generator of a strongly continuous analytic contraction semigroup $\{T_t, t\geq 0\}$ on $L^2(\Omega,m)$.
\end{teo}
\begin{proof} The analyticity of $\{T_t\}$ follows from Proposition \ref{coer} (see Theorem 6.2, Chapter 4 in~\cite{showal}). The contraction property follows from Lumer-Phillips Theorem (see Theorem 4.3, Chapter 1 in~\cite{pazy}). The strong continuity follows from Theorem 1.3.1 in~\cite{fukush}.
\end{proof}

\section{A priori estimates in weighted Sobolev spaces}\label{secapriori}
\setcounter{equation}{0}

\noindent In this section we prove a priori estimates for the solution of problem \label{pbformale}. We stress the fact that the key issue is to prove that $u\in H^2(K_n)$, which does not follow as in the case of local Venttsel' problem (see \cite{lanver2}).\\
We consider the problem formally stated as follows: for every $t\in [0,T]$
\begin{equation}\label{pbformale}
\begin{cases}
\frac{du}{dt}=\Delta u+f &\text{in $\Omega$,}\\[2mm]
-\Delta_{K_n} u=-\frac{\partial u}{\partial\nu}-bu-\theta_2(u)+f-\frac{du}{dt}\quad &\text{on $K_n$},\\[2mm]
u(0,x)=u_0(x) &\text{in $\overline\Omega$}.
\end{cases}
\end{equation}
where $f$ and $u_0$ are given functions in suitable spaces.

\begin{teo}\label{aprioriest} Let $f\in L^2(\Omega,m)$ and $b$ and $\theta_2 (u)$ be as defined above. For every $t\in [0,T]$ let $u\in V^2_\sigma(\Omega,K_n)$ be a solution of problem \eqref{pbformale}. Then there exists a positive constant $C=C(\sigma)$ such that 
\begin{equation}\label{stima5}
\|u\|^2_{H^1(\Omega)}+\|r^\sigma D^2 u\|^2_{L^2(\Omega)}+\|u\|^2_{H^2(K_n)}\leq C(\sigma)\left(\|u\|^2_{L^2(K_n)}+\|f\|^2_{L^2(\Omega,m)}+\left\|\frac{du}{dt}\right\|^2_{L^2(\Omega,m)}\right),
\end{equation}
provided 
\begin{equation}\label{boundsigma}
\frac{1}{4}<\sigma<\frac{1}{2}.
\end{equation}
\end{teo}

\begin{proof}
We adapt the proof of Theorem 2.1 in \cite{crelannazver} to our case. We use the Munchhausen trick. We start by assuming that $\frac{\partial u}{\partial\nu}$ and $\theta_2 (u)$ belong to $L^2(K_n)$, hence the right-hand side of the second equation in \eqref{pbformale} belongs to $L^2(K_n)$. Then the following estimate holds:
\begin{equation}\label{stima1}
\|u\|^2_{H^2(K_n)}\leq C\left(\left\|\frac{\partial u}{\partial\nu}\right\|^2_{L^2(K_n)} + \|u\|^2_{L^2(K_n)}+\|\theta_2 (u)\|^2_{L^2(K_n)}+\|f\|^2_{L^2(K_n)}+\left\|\frac{du}{dt}\right\|^2_{L^2(K_n)}\right).
\end{equation}
First we estimate $\|\theta_2(u)\|^2_{L^2(K_n)}$. We note that since $u\in H^2(K_n)$ it follows in particular that $\theta_2 (u)\in H^{\beta} (K_n)$ with $\beta<1/2$. This can be seen by using the definition of Sobolev space by the Fourier transform $\mathcal{F}$:
\begin{center}
$H^s (\mathbb{R}) =\{v\in\mathcal{S}'\,\mid\, (1+|\xi|^2)^{s/2}\mathcal{F}[v]\in L^2(\mathbb{R})\}$,
\end{center}
where $\mathcal{S}'$ is the space of tempered distributions.\\
We rectify the boundary $K_n$. Our function $u$ belongs to $H^2(K_n)$, then it is piecewise $H^2$ on each segment $M$ of $K_n$. Roughly speaking, the singularity in the corner can be described by a function like $|x|$ (in order to consider the Fourier transform of $|x|$, we have to mollify the function outside of a neighborhood of the singularity). The Fourier transform of $v(x):=|x|$ behaves like $\xi^{-2}$, hence it belongs to $H^\alpha (\R)$ if and only if $\alpha<3/2$. Since the functional $\theta_2(u)$ behaves like the fractional Laplacian $(-\Delta)^\frac{1}{2}$, then $\theta_2(v)$ belongs to $H^{\alpha-1} (\R)$ with $\alpha<3/2$. This implies that $\theta_2(u)\in H^\beta (K_n)$ with $\beta<1/2$ and
\begin{equation}\notag
\|\theta_2(u)\|^2_{H^{\beta}(K_n)}\leq C\|u\|^2_{H^2(K_n)},
\end{equation}
where $C$ depends on $\beta$.\\
We fix $\beta<1/2$. From the compact embedding of $H^{\beta} (K_n)$ in $L^2(K_n)$ we deduce that for every $\varepsilon>0$ there exists a constant $C(\varepsilon)$ such that
\begin{equation}\notag
\|\theta_2(u)\|^2_{L^2(K_n)}\leq\varepsilon\|\theta_2(u)\|^2_{H^{\beta} (K_n)}+C(\varepsilon)\|\theta_2(u)\|^2_{H^{-\frac{1}{2}}(K_n)},
\end{equation}
see Lemma 6.1, Chapter 2 in~\cite{necas}. Similarly, we have
\begin{equation}\notag
\|\theta_2(u)\|^2_{H^{-\frac{1}{2}}(K_n)}\leq C\|u\|^2_{H^\frac{1}{2}(K_n)}\leq\varepsilon\|u\|^2_{H^2(K_n)}+C(\varepsilon)\|u\|^2_{L^2(K_n)}.
\end{equation}
Therefore we obtain the following estimate using \eqref{stima1}:
\begin{equation}\notag
\|u\|^2_{H^2 (K_n)} \leq C\left(\left\|\frac{\partial u}{\partial\nu}\right\|^2_{L^2(K_n)} +\|f\|^2_{L^2(K_n)}+\left\|\frac{du}{dt}\right\|^2_{L^2(K_n)}+\varepsilon\|u\|^2_{H^{2}(K_n)}+C(\varepsilon)\|u\|^2_{L^2(K_n)}\right).
\end{equation}
By choosing $\varepsilon$ sufficiently small we obtain
\begin{equation}\label{intermedia}
\|u\|^2_{H^2(K_n)}\leq C\left(\left\|\frac{\partial u}{\partial\nu}\right\|^2_{L^2(K_n)} + \|u\|^2_{L^2(K_n)}+\|f\|^2_{L^2(K_n)}+\left\|\frac{du}{dt}\right\|^2_{L^2(K_n)}\right),
\end{equation}
with $C$ independent of $\varepsilon$.\\
We now estimate $\left\|\frac{\partial u}{\partial\nu}\right\|^2_{L^2(K_n)}$. 
We consider a smooth function $U$ on $\overline\Omega$ which is linear near the corners of $K_n$ and such that $(u-U)(P)=\nabla_s(u-U)(P)=0$ in every vertex $P$ of $K_n$. 

\noindent If we consider the function $v=u-U$, from Hardy inequality applied on each segment of $K_n$ (see~\cite{hardy}) we obtain that $v\in H^2_{\gamma=0} (K_n)$ and $v\in H^\frac{3}{2}_\sigma(K_n)$, with
\begin{equation}\label{stimav}
\|v\|_{H^\frac{3}{2}_\sigma(K_n)}\leq C\|u\|_{H^2(K_n)}.
\end{equation}

\noindent Now we consider $v$ as the solution of the Dirichlet problem
\begin{equation}\label{problemav}
\begin{cases}
-\Delta v=f-\frac{du}{dt}+\Delta U\in L^2 (\Omega),\\[2mm]
v|_{K_n}\in H^\frac{3}{2}_\sigma (K_n).
\end{cases}
\end{equation}
We note that, due to our hypothesis on $\sigma$, in particular $f-\frac{du}{dt}\in L^2_\sigma (\Omega)$. Hence, from Theorem 3.1, Chapter 2 in~\cite{nazplam} (with $l=0$) it follows that $v\in H^2_{\sigma} (\Omega)$ if $|\sigma-1|<3/4$ and, 
from the hypothesis on the function $U$ and \eqref{stimav}, the following estimate holds:
\begin{equation}\label{stimaU+v}
\begin{split}
\|u\|^2_{H^1(\Omega)}+\|r^\sigma D^2 u\|^2_{L^2(\Omega)}\leq C(\sigma)\left(\left\|r^\sigma\left(f-\frac{du}{dt}\right)\right\|^2_{L^2(\Omega)}+\|u\|^2_{H^2(K_n)}\right)\\
\leq C(\sigma)\left(\left\|\frac{du}{dt}\right\|^2_{L^2(\Omega)}+\|f\|^2_{L^2(\Omega)}+\|u\|^2_{H^2(K_n)}\right).
\end{split}
\end{equation}

\noindent Now we introduce the following trace operator:
\begin{equation}\notag
u\longrightarrow\frac{\partial u}{\partial\nu}\Big|_{K_n}=\eta_\delta\frac{\partial u}{\partial\nu}\Big|_{K_n}+(1-\eta_\delta)\frac{\partial u}{\partial\nu}\Big|_{K_n}=:\mathcal{K}_1(\delta)u+\mathcal{K}_2(\delta)u,
\end{equation}
where $\eta_\delta$ is a suitable cutoff function near the vertices. We remark that the operator $\mathcal{K}_1(\delta)\colon H^2_\sigma(\Omega)\to L^2(\partial\Omega)$ is compact and the following estimate holds for $\mathcal{K}_2(\delta)$:
\begin{equation}\notag
\|\mathcal{K}_2(\delta)u\|^2_{L^2(\partial\Omega)}\leq C(\sigma)\delta^{\frac{1}{2}-\sigma}\left(\left\|\frac{du}{dt}\right\|^2_{L^2(\Omega)}+\|f\|^2_{L^2(\Omega)}+\|u\|^2_{H^2(K_n)}\right).
\end{equation}
This in turn implies that $\sigma<1/2$. Hence, by choosing $\delta$ sufficiently small, it follows that
\begin{equation}\notag
\left\|\frac{\partial u}{\partial\nu}\right\|^2_{L^2(K_n)}\leq\varepsilon\left(\left\|\frac{du}{dt}\right\|^2_{L^2(\Omega)}+\|f\|^2_{L^2(\Omega)}+\|u\|^2_{H^2(K_n)}\right)+C(\varepsilon,\sigma)\|u\|^2_{L^2(K_n)}.
\end{equation}
Substituting the above inequality into \eqref{intermedia} we obtain
\begin{center}
$\displaystyle\|u\|^2_{H^2(K_n)}\leq C\left(\varepsilon\left(\left\|\frac{du}{dt}\right\|^2_{L^2(\Omega)}+\|f\|^2_{L^2(\Omega)}+\|u\|^2_{H^2(K_n)}\right)+C(\varepsilon,\sigma)\|u\|^2_{L^2(K_n)}+\|f\|^2_{L^2(K_n)}+\left\|\frac{du}{dt}\right\|^2_{L^2(K_n)}\right)$.
\end{center}
By choosing $\varepsilon$ sufficiently small we obtain
\begin{equation}\label{intermedia2}
\|u\|^2_{H^2(K_n)}\leq C(\sigma)\left(\|f\|^2_{L^2(\Omega,m)}+\|u\|^2_{L^2(K_n)}+\left\|\frac{du}{dt}\right\|^2_{L^2(\Omega,m)}\right)
\end{equation}
and, taking into account \eqref{stimaU+v}, we get the thesis.
\end{proof}

\section{Existence and uniqueness results}\label{sec3}
\setcounter{equation}{0}

\noindent We now consider the following abstract Cauchy problem, for $T>0$ fixed:
\begin{equation}\label{astratto}
(P)
\begin{cases}
u'(t)=Au(t)+f(t)\quad\text{for $t\in[0,T]$},\\
u(0)=u_0.
\end{cases}
\end{equation}
From semigroup theory we get the following existence and uniqueness result.
\begin{teo}\label{strict}
Let $\alpha\in (0,1)$, $f\in C^{0,\alpha} ([0,T];L^2 (\Omega,m))$ and $u_0\in D(A)$. We define
\begin{equation}\label{soluz}
 \displaystyle u(t)=T_t\,u_0+\int_0^t T_{t-\tau}\,f(\tau)\, d\tau.
\end{equation}
Then $u$ defined in \eqref{soluz} is the unique \emph{strict} solution of problem $(P)$, i.e. a function $u$ such that $u'(t)=Au(t)+f(t)$ for all $t\in [0,T]$, $u(0)=u_0$ and
\begin{center}
$u\in C^1 ([0,T]; L^2(\Omega,m))\cap C^0 ([0,T]; D(A))$.
\end{center}
Moreover the following estimate holds:
\begin{center}
$\|u\|_{C^1([0,T]; L^2(\Omega,m))}+\|u\|_{C^0([0,T];D(A))}\leq C\left(\|u_0\|_{D(A)}+\|f\|_{C^{0,\alpha}([0,T];L^2(\Omega,m))}\right)$,
\end{center}
where $C$ is a constant independent of $n$.
\end{teo}

\noindent For the proof see Theorem 4.3.1 in~\cite{lunardi}.
\begin{oss} If we suppose $u_0\in\overline{D(A)}$ in Theorem \ref{strict}, then the solution $u$ of equation \eqref{astratto} is \emph{classical}. In addition to that, we have that $u\in C^{1,\alpha} ([\varepsilon,T]; L^2(\Omega,m))\cap C^{0,\alpha} ([\varepsilon,T]; D(A))$ for every $\varepsilon\in (0,T)$.
\end{oss}

\noindent We now give the strong formulation of the abstract Cauchy problem $(P)$.
\begin{teo} Let $u$ be the unique strict solution of \eqref{astratto}. Then, for every $t\in [0,T]$, it holds that
\begin{equation}\label{pbforte}
\begin{cases}
\frac{du}{dt}(t,x)=\Delta u(t,x)+f(t,x) &\text{for a.e. $x=(x_1,x_2)\in\Omega$,}\\[2mm]
\frac{du}{dt}=-\frac{\partial u}{\partial\nu}+\Delta_{K_n} u-bu-\theta_2(u)+f\quad &\text{in $H^{-\frac{1}{2}} (K_n)$},\\[2mm]
u(0,x)=u_0(x) &\text{in $\overline\Omega$}.
\end{cases}
\end{equation}
\end{teo}
\begin{proof} For every fixed $t\in[0,T]$, we multiply the first equation in \eqref{astratto} by a test function $\varphi\in C^{\infty}_0 (\Omega)$ and then we integrate on $\Omega$. Then, by using \eqref{corr}, we obtain
\begin{center}
$\displaystyle\int_\Omega \frac{du}{dt}\,\varphi\,dx=\int_{\Omega} Au\,\varphi\,dx+\int_{\Omega} f\,\varphi\,dx=-E(u,\varphi)+\int_{\Omega} f\,\varphi\,dx$.
\end{center}
Since $\varphi$ has compact support in $\Omega$, after integrating by parts, we get
\begin{equation}\label{forte}
\frac{du}{dt}=\Delta u+f\quad\text{in $(C_0^{\infty} (\Omega))'$},
\end{equation}
then, by density, equation \eqref{forte} holds in $L^2(\Omega)$, so it holds for a.e. $x\in\Omega$. We remark that from this it follows that, for each fixed $t\in [0,T]$, $u\in V(\Omega):=\{u\in H^1 (\Omega)\,|\,\Delta u\in L^2(\Omega)\}$, where $\Delta u$ has to be intended in the distributional sense. Hence, we can apply Green formula for Lipschitz domains (see~\cite{baiocap}) which yields in particular that $\frac{\partial u}{\partial\nu}\in H^{-\frac{1}{2}} (K_n)$:
\begin{center}
$\displaystyle\int_{\Omega}\nabla u\,\nabla v\,dx=\left\langle\frac{\partial u}{\partial\nu},v\right\rangle-\int_{\Omega}\Delta u\,v\,dx$.
\end{center}
We now come to the dynamical boundary condition. Let $v\in V(\Omega,K_n)$. We take the scalar product in $L^2(\Omega,m)$ between the first equation in \eqref{astratto} and $v$, so we obtain
\begin{equation}\label{formvariaz}
\left(\frac{du}{dt},v\right)_{L^2(\Omega,m)}=(Au,v)_{L^2(\Omega,m)}+(f,v)_{L^2(\Omega,m)}.
\end{equation}
Then, by using again \eqref{corr}, we have that
\begin{center}
$\displaystyle\int_\Omega \frac{du}{dt}\,v\,dx+\int_{K_n} \frac{du}{dt}\,v\,ds=-\int_{\Omega}\nabla u\,\nabla v\,dx-\int_{K_n}\nabla_s u\,\nabla_s v\,ds-\int_{K_n} b\,u\,v\,ds-\langle\theta_2 (u),v\rangle+\int_{\Omega} f\,v\,dx+\int_{K_n} f\,v\,ds$.
\end{center}
Now, using Green formula for Lipschitz domains and the fact that equation \eqref{forte} holds a.e. in $\Omega$, we obtain $\forall\,v\in V(\Omega,K_n)$ and for each $t\in [0,T]$
\begin{equation}\label{bordo}
\int_{K_n} \frac{du}{dt}\,v\,ds=-\left\langle\frac{\partial u}{\partial\nu},v\right\rangle-\int_{K_n}\nabla_s u\,\nabla_s v\,ds-\int_{K_n} b\,u\,v\,ds-\langle\theta_2 (u),v\rangle+\int_{K_n} f\,v\,ds.
\end{equation}
Since $H^1(K_n)$ is dense in $H^{\frac{1}{2}} (K_n)$ (see \cite{baiocap}), we deduce that the boundary condition
\begin{equation}\label{bordo2}
\frac{du}{dt}-\Delta_{K_n}\,u=-\frac{\partial u}{\partial\nu}-bu-\theta_2(u)+f
\end{equation}
holds in $H^{-\frac{1}{2}} (K_n)$.
\end{proof}

\noindent We now prove a better regularity in space of the solution of problem \eqref{pbforte}.
\begin{teo}\label{esistenzagiusta} Let $\sigma$, $f$, $b$ and $\theta_2 (u)$ be as in Theorem \ref{aprioriest}. Then for every $t\in [0,T]$ the solution of problem \eqref{pbforte} belongs to $V^2_\sigma (\Omega,K_n)$, and the following inequality holds:
\begin{equation}\label{regolarita}
\|u\|^2_{H^1(\Omega)}+\|r^\sigma D^2 u\|^2_{L^2(\Omega)}+\|u\|^2_{H^2(K_n)}\leq C\left(\|f\|^2_{L^2(\Omega,m)}+\left\|\frac{du}{dt}\right\|^2_{L^2(\Omega,m)}\right),
\end{equation}
where $C$ depends on $\sigma$.
\end{teo}

\begin{proof} We rewrite problem \eqref{pbforte} as
\begin{equation}\label{pbforteell}
\begin{cases}
-\Delta u=f-\frac{du}{dt} &\text{in $\Omega$,}\\[2mm]
-\Delta_{K_n} u=-\frac{\partial u}{\partial\nu}-bu-\theta_2(u)+f-\frac{du}{dt}\quad &\text{on $K_n$}.
\end{cases}
\end{equation}
We note that, for every $t\in [0,T]$, $f-\frac{du}{dt}\in L^2(\Omega,m)$. Hence, from elliptic regularity results of Theorem 3.3 in \cite{crelannazver}, we deduce that for every $t\in [0,T]$ $u\in V^2_\sigma (\Omega,K_n)$ and \eqref{regolarita} holds.
\end{proof}


\section{Regularity results in fractional Sobolev spaces}\label{sec4}
\setcounter{equation}{0}

\noindent We now prove some regularity results for the strict solution $u$ of \eqref{pbforte}. 


\begin{teo}\label{regolare} Let $u$ be the solution of problem \eqref{pbforte}. Then, for every fixed $t\in [0,T]$, $u\in H^s (\Omega)$ for $s<\frac{7}{4}$.
\end{teo}

\begin{proof} Let us consider for every fixed $t\in [0,T]$ the weak solutions $w$ and $\hat{w}$ in $H^1 (\Omega)$ of the following auxiliary problems:
\begin{equation}\label{aux1}
\begin{cases}
\Delta\hat{w}=0\quad &\text{in}\,\Omega\\
\hat{w}=u &\text{on}\, K_n,
\end{cases}
\end{equation}
\begin{equation}\label{aux2}
\begin{cases}
-\Delta w=-\frac{du}{dt}+f\quad &\text{in}\,\Omega\\
w=0 &\text{on}\, K_n.
\end{cases}
\end{equation}
We point out that the regularity of the solution $u$ of problem \eqref{pbforte} follows from the regularity of $\hat{w}$ and $w$ since
\begin{equation}\label{relaz}
u=\hat{w}+w.
\end{equation}
From Theorems 2 and 3 in~\cite{kenig}, it follows that
\begin{equation}\label{jerken}
\displaystyle\frac{\partial\hat{w}}{\partial\nu}\in L^2 (K_n).
\end{equation}
As to the solution $w$ of problem \eqref{aux2}, we remark that the right-hand side of the first equation belongs to $L^2 (\Omega)$. From Kondrat'ev regularity results for the solutions of elliptic problems in corners (see \cite{kond}), since $f-\frac{du}{dt}\in L^2_{\mu}(\Omega)$ for $\mu>1/4$ (taking into account that the angles in $\Omega$ have opening equal to $\frac{\pi}{3}$ or $\frac{4\pi}{3}$), we get 
\begin{equation}\label{sob}
\|\delta^{\mu}D^{\alpha} w\|^2_{L^2(\Omega)}\leq C(\mu,n)\left\|f-\frac{du}{dt}\right\|_{L^2(\Omega)}\quad\text{for $|\alpha|=2$ and $\mu>\frac{1}{4}$},
\end{equation}
where $\delta$ denotes the distance from the boundary. Now, by using Proposition 4.15 in~\cite{jerison}, we have
\begin{center}
$\displaystyle\|w\|_{H^{2-\mu}(\Omega)}\leq c\left\{\|\delta^{\mu}\sum_{|\alpha|=2} D^{\alpha} w\|^2_{L^2(\Omega)}+\|w\|_{H^1(\Omega)}\right\}^{\frac{1}{2}}$
\end{center}
and from \eqref{sob} it follows that $w\in H^s (\Omega)$ for $s<7/4$.\\
We now prove that $\hat{w}$ has the same regularity of $w$. Since $u\in H^2(K_n)$, in particular $u$ belongs to $H^{\frac{3}{2}}(K_n)$. Then from the trace theorem (Proposition \ref{traccia}) there exists a function $\tilde{u}$ which belongs to $H^2(\Omega)$ and such that $\gamma_0\tilde{u}=u$. If we consider then the function $\tilde{w}=\hat{w}-\tilde{u}$, this function belongs to $H^1(\Omega)$ and it is the weak solution of the auxiliary problem
\begin{equation}\label{aux3}
\begin{cases}
\Delta\tilde{w}=-\Delta\tilde{u} \quad &\text{in $\Omega$} \\
\tilde{w}=0 &\text{on $K_n$}.
\end{cases}
\end{equation}
Analogously, since $\Delta\tilde{u}$ belongs to $L^2(\Omega)$, we obtain that $\tilde{w}$ belongs to $H^2_{\mu}(\Omega)$ for $\mu>\frac{1}{4}$. This in particular implies that $\hat{w}$ belongs to $H^2_{\mu}(\Omega)$ for $\mu>\frac{1}{4}$, hence from Proposition 4.15 in~\cite{jerison} it follows that $u\in H^s(\Omega)$ for $s<\frac{7}{4}$.
\end{proof}

\begin{oss} By proceeding as in Theorem 4.2 in~\cite{lanmos}, with the obvious changes, we can prove that $u\in H^2_{\mu} (\Omega)$ for $\mu>1/4$, with weight given by the distance from the reentrant vertices (i.e. the vertices of the prefractal curve $K_n$).
\end{oss}

\begin{oss} From Theorem \ref{regolare}, we have that the solution of problem $(P)$ is the solution of the following problem: for every $t\in [0,T]$,
\begin{equation}\label{pbforte2}
\begin{cases}
\frac{du}{dt}=\Delta u+f &\text{in $L^2(\Omega)$,}\\[2mm]
\frac{du}{dt}=-\frac{\partial u}{\partial\nu}+\Delta_{K_n} u-bu-\theta_2(u)+f\quad &\text{in $L^2 (K_n)$},\\[2mm]
u(0,x)=u_0(x) &\text{a.e. $x\in\overline\Omega$}.
\end{cases}
\end{equation}
where $\Delta_{K_n}$ is the piecewise tangential Laplacian.
\end{oss}

\medskip
\noindent Since
\begin{equation}\label{eps}
u\in H^{\frac{7}{4}-\varepsilon} (\Omega)\quad\forall\,\varepsilon>0,
\end{equation}
from Sobolev embedding theorems (see Theorem 1.4.5.2 in~\cite{grisvard}), we have that
\begin{center}
$u\in C^{0,\delta} (\overline{\Omega})$ with $\delta=\frac{3}{4}-\varepsilon$.
\end{center}
We remark that, just knowing that $u\in H^1(K_n)$, from Sobolev embedding theorems we deduce that $u\in C^{0,\frac{1}{2}} (K_n)$. From Theorem \ref{regolare} we obtain a better regularity for the solution $u$ of problem \eqref{pbforte2}.\\
Moreover, since $u\in H^1(K_n)$ and $\nabla u\in H^{\frac{3}{4}-\varepsilon} (\Omega)$, from the trace theorem we have $\nabla u|_{K_n}\in H^{s_1} (K_n)$ for $0<s_1<\frac{1}{4}$.
\begin{oss}\label{sompropsol} For the reader's convenience, we now summarize the main regularity properties of the solution of problem \eqref{pbforte2} which will turn crucial in order to prove the a priori error estimates in Section \ref{sec5}: for every $t\in [0,T]$
\begin{center}
$\begin{cases}
\frac{du}{dt}=\Delta u+f &\text{in $L^2(\Omega)$,}\\[2mm]
\frac{du}{dt}=-\frac{\partial u}{\partial\nu}+\Delta_{K_n} u-bu-\theta_2(u)+f\quad &\text{in $L^2 (K_n)$},\\[2mm]
u(0,x)=u_0(x) &\text{a.e. $x\in\overline\Omega$}.
\end{cases}$
\end{center}
$u$ belongs to $H^2_\mu (\Omega)$ for $\mu>1/4$, it is H\"older continuous on $\overline\Omega$ with $\delta=\frac{3}{4}-\varepsilon$ and its trace belongs to $H^2(K_n)$.
\end{oss}

\section{A priori error estimates}\label{sec5}
\setcounter{equation}{0}

\noindent We start this section by proving some a priori estimates for the solution $u$ of problem \eqref{pbforte2}.

\begin{prop}\label{stabilita} Let $u$ be the solution of problem \eqref{pbforte2}. Then, for every fixed $t\in [0,T]$ it holds:
\begin{equation}\label{stima2}
\|u(t)\|^2_{L^2(\Omega,m)}+\int_0^t \|u(\tau)\|_{V(\Omega,K_n)}^2\,d\tau\leq C\left(\|u_0\|^2_{L^2(\Omega,m)}+\int_0^t \|f(\tau)\|^2_{L^2(\Omega,m)}\,d\tau\right),
\end{equation}
where $C$ is a constant depending on $t$ and on the coercivity constant of $E$.
\end{prop}

\begin{proof} We write the weak formulation of problem \eqref{pbforte2}: for each $t\in [0,T]$,
\begin{center}
$\displaystyle \left(\frac{du}{dt}(t),v\right)_{L^2(\Omega,m)}+E(u(t),v)=(f(t),v)_{L^2(\Omega,m)}\quad\forall\,v\in V(\Omega,K_n)$.
\end{center}
If we choose $v=u(t)$, thanks to the coercivity of $E$ and Young and Gronwall inequalities, the thesis follows.
\end{proof}

\begin{teo} Let $u$ be the solution of problem \eqref{pbforte2}. Then it holds that
\begin{equation}\label{stima1}
\begin{split}
&\displaystyle\int_0^T \left\|\frac{d u}{dt}(\tau)\right\|^2_{L^2(\Omega,m)}\,d\tau+\sup_{t\in [0,T]}\|u(t)\|^2_{V(\Omega,K_n)}\\
&\leq\frac{1}{\min\{1,\bar{C}\}}\left(\tilde{C}\|u_0\|^2_{V(\Omega,K_n)}+\int_0^T \|f(\tau)\|^2_{L^2(\Omega,m)}\,d\tau\right),
\end{split}
\end{equation}
where $\bar{C}$ is the coercivity constant of $E$, while the constant $\tilde{C}$ depends on $n$.
\end{teo}

\begin{proof} In order to prove this estimate, we use the \emph{Faedo-Galerkin} method. Let $\{\phi_j\}_{j=1}^{\infty}$ be a complete orthonormal basis of $V(\Omega,K_n)$, and $V^N=\spann\{\phi_1,\dots,\phi_N\}$. We define $u_N\colon [0,T]\to V^N$ in the following way:
\begin{equation}\label{defu}
u_N(t):=\sum_{j=1}^N d_N^j (t)\,\phi_j,
\end{equation}
where we select the coefficients $d_N^j (t)$, for $t\in [0,T]$ and $j=1,\dots,N$ such that
\begin{equation}\label{fg1}
d_N^j (0)=(u_0,\phi_j)_{L^2(\Omega,m)}
\end{equation}
and
\begin{equation}\label{fg}
\left(\frac{d u_N}{dt}(t),\phi_j\right)_{L^2(\Omega,m)}+E(u_N (t),\phi_j)=(f(t),\phi_j)_{L^2(\Omega,m)}\quad\forall\,j=1,\dots,N.
\end{equation}
The existence and uniqueness of a function $u_N$ of the form \eqref{defu} follows from standard ODE theory.\\
Now we multiply equation \eqref{fg} by $(d_N^j)'(t)$. Then, by taking the sum on $j=1,\dots,N$, we obtain
\begin{equation}\label{fg2}
\left(\frac{d u_N}{dt}(t),\frac{d u_N}{dt}(t)\right)_{L^2(\Omega,m)}+E\left(u_N (t),\frac{d u_N}{dt}(t)\right)=\left(f(t),\frac{d u_N}{dt}(t)\right)_{L^2(\Omega,m)}.
\end{equation}
We now observe that, by using Cauchy-Schwartz and Young inequalities, we obtain
\begin{center}
$\displaystyle \left(f(t),\frac{d u_N}{dt}(t)\right)_{L^2(\Omega,m)}\leq\|f(t)\|_{L^2(\Omega,m)}\,\left\|\frac{d u_N}{dt}(t)\right\|_{L^2(\Omega,m)}\leq\frac{1}{2}\|f(t)\|^2_{L^2(\Omega,m)}+\frac{1}{2}\left\|\frac{d u_N}{dt}(t)\right\|^2_{L^2(\Omega,m)}$.
\end{center}
Next, we point out that
\begin{center}
$\displaystyle E\left(u_N,\frac{d u_N}{dt}\right)=\int_{\Omega} \nabla u_N\,\nabla\left(\frac{d u_N}{dt}\right)\,dx+\int_{K_n} \nabla_s u_N\,\nabla_s\left(\frac{d u_N}{dt}\right)\,ds+\int_{K_n} b\,u_N\,\frac{d u_N}{dt}\,ds+\iint_{K_n\times K_n}\frac{(u_N(x)-u_N(y))(\frac{d u_N}{dt}(x)-\frac{d u_N}{dt}(y))}{|x-y|^2}\,ds(x)\,ds(y)=\frac{1}{2}\frac{d}{dt} E[u_N]$.
\end{center}
Then, following these two remarks, \eqref{fg2} can be written in this way:
\begin{equation}\label{fg3}
\frac{1}{2}\left\|\frac{d u_N}{dt}(t)\right\|^2_{L^2(\Omega,m)}+\frac{1}{2}\frac{d}{dt} E[u_N]\leq\frac{1}{2}\|f(t)\|^2_{L^2(\Omega,m)}.
\end{equation}
Now, integrating \eqref{fg3} on $(0,t)$, using the coercivity of $E$ and taking the supremum on $[0,T]$ we obtain
\begin{equation}\label{stimaN}
\begin{split}
&\int_0^T \left\|\frac{d u_N}{dt}(\tau)\right\|^2_{L^2(\Omega,m)}\,d\tau+\sup_{t\in [0,T]}\|u_N(t)\|^2_{V(\Omega,K_n)}\\
&\leq\frac{1}{\min\{1,\bar{C}\}}\left(E[u_N(0)]+\int_0^T \|f(\tau)\|^2_{L^2(\Omega,m)}\,d\tau\right),
\end{split}
\end{equation}
where $\bar{C}$ is the coercivity constant of $E$. We now point out that, since $\displaystyle\langle\theta_2 (u),u\rangle\leq\tilde{C_1}(n)\|u\|^2_{H^1(K_n)}$, it follows that
\begin{center}
$\displaystyle E[u_N(0)]\leq\tilde{C}_2(n)\|u_0\|^2_{V(\Omega,K_n)}$.
\end{center}
Hence we get
\begin{equation}\label{stimaN2}
\begin{split}
&\int_0^T \left\|\frac{d u_N}{dt}(\tau)\right\|^2_{L^2(\Omega,m)}\,d\tau+\sup_{t\in [0,T]}\|u_N(t)\|^2_{V(\Omega,K_n)}\\
&\leq\frac{1}{\min\{1,\bar{C}\}}\left(\tilde{C}_2(n)\|u_0\|^2_{V(\Omega,K_n)}+\int_0^T \|f(\tau)\|^2_{L^2(\Omega,m)}\,d\tau\right),
\end{split}
\end{equation}
i.e. the thesis for $u_N$.\\
Now we want to prove the estimate for $u$. At first we observe that from \eqref{stimaN} it follows that $\frac{d u_N}{dt}\in L^2([0,T];L^2(\Omega,m))$ and $u_N\in L^{\infty} ([0,T];V(\Omega,K_n))$. Then, there exists a subsequence $\{\frac{d u_{N_l}}{dt}\}_{l=1}^{\infty}$ of $\{\frac{d u_N}{dt}\}$ such that $\frac{d u_{N_l}}{dt}\rightharpoonup\frac{du}{dt}$ in $L^2([0,T];L^2(\Omega,m))$ when $l\to+\infty$. Then we conclude as in Theorem 5, Section 7.1.3 in~\cite{evans}.
\end{proof}

\begin{corol} Let $u$ be the solution of problem \eqref{pbforte2}. Then it holds that
\begin{equation}\label{stima3}
\begin{split}
\int_0^T &\left(\left\|\frac{du}{dt}(\tau)\right\|^2_{L^2(\Omega,m)}+\|u(\tau)\|^2_{H^2_{\mu}(\Omega)}\right)\,d\tau+\sup_{t\in [0,T]}\|u(t)\|^2_{V(\Omega,K_n)}\\
& \leq C\left(\|u_0\|^2_{V(\Omega,K_n)}+\int_0^T \|f(\tau)\|^2_{L^2(\Omega,m)}\,d\tau\right),
\end{split}
\end{equation}
where $C$ is a constant depending on $\mu$, $n$, $T$ and the coercivity constant of $E$.
\end{corol}

\begin{proof} Estimate \eqref{stima3} follows from \eqref{sob}, \eqref{stima2} and \eqref{stima1}.
\end{proof}

\medskip

\noindent We now focus our attention on the numerical approximation of problem $(P)$. It will be carried out in two steps. In the former one we discretize by a Galerkin method the space variable only. We obtain an a priori error estimate for the semi-discrete solution. In the latter one we consider the fully discretized problem by a finite difference approach on the time variable.
\begin{figure}
	\centering
		\includegraphics[width=0.40\textwidth]{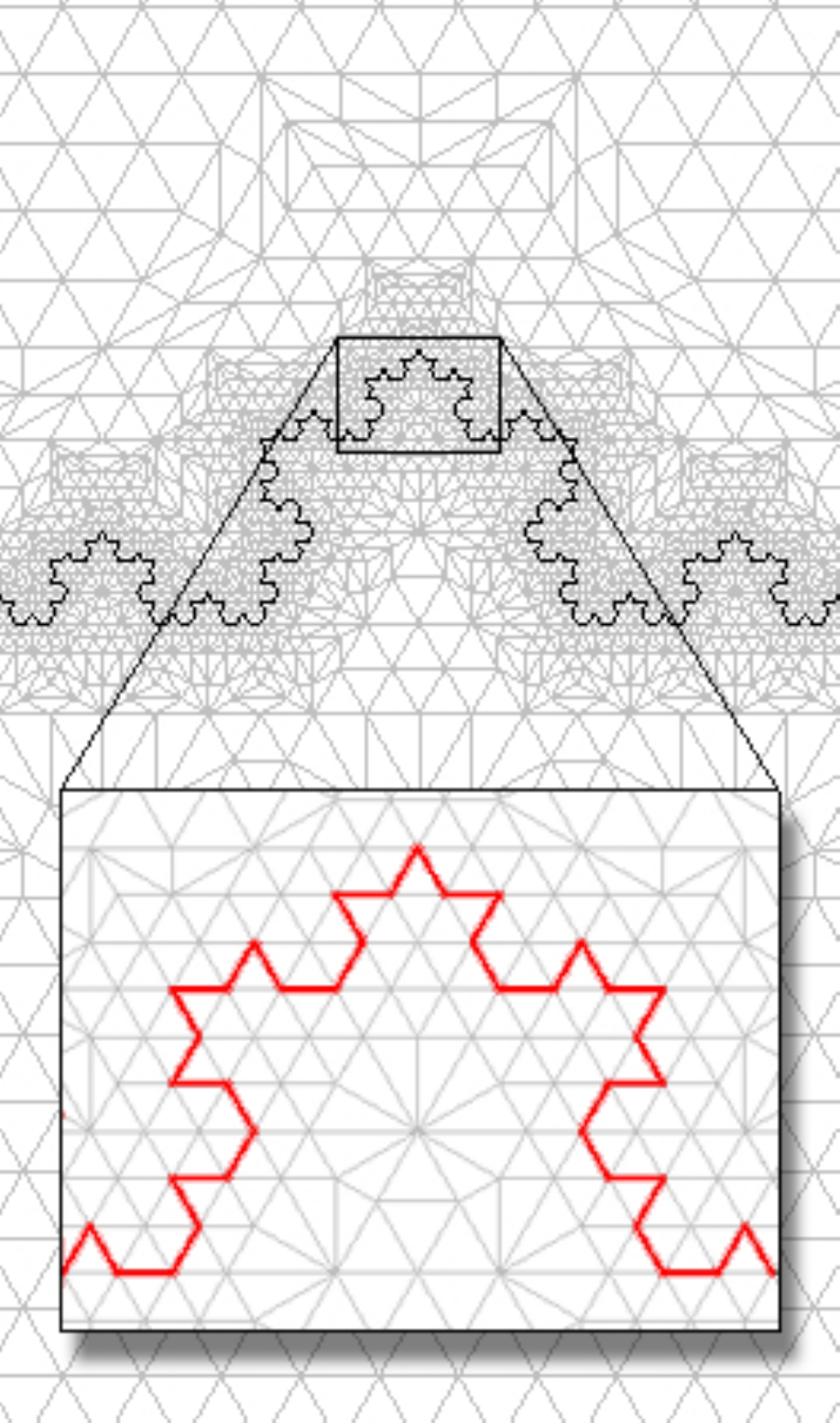}
	\caption{A zoom of the mesh produced by our algorithm.}
	\label{mesh}
\end{figure}
\noindent In order to obtain optimal a priori error estimates, we will use a suitable mesh (see \cite{cefalolancia} and \cite{lancef}) which is compliant with the so-called \emph{Grisvard conditions}. More precisely, since the solution $u$ is not in $H^2(\Omega)$, one has to use a suitable mesh refinement process which can guarantee an optimal rate of convergence. Here we do not give details on the mesh algorithm neither on its properties, we refer to \cite{cefalolancia} and \cite{lancef} for the details.\\
Our mesh refinement process generates a \emph{conformal} and \emph{regular} family of triangulations $\{T_{n,h}\}$, where $h=\max\{diam(S),\,S\in T_{n,h}\}$ is the size of the triangulation which is also compliant with the Grisvard conditions. We denote by $X_h:=\{v\in C^0 (\Omega)\,:\,v|_{S}\in\mathbb{P}_1\,\forall\, S\in T_{n,h}\}$, where $\mathbb{P}_1$ denotes the space of polynomial functions of degree one. Let $I_h\colon H^2_{\mu}(\Omega)\to X_h$, with $\mu>1/4$, be the $X_h$-interpolant operator, defined as:
\begin{center}
$I_h(u)|_{S}\in\mathbb{P}_1$ for every $S\in T_{n,h}$ and $I_h(u)=u$ at any vertex of any $S\in T_{n,h}$.
\end{center}
We note that $I_h$ is well defined since $u$ is in particular continuous on $\overline\Omega$ (see Remark \ref{sompropsol}). Moreover we suppose that the family of triangulations $\{T_{n,h}\}$ satisfies the hypothesis of Theorem 4.2 in~\cite{lancef}. We define the finite dimensional space of piecewise linear functions
\begin{center}
$X_{n,h}:=\{v\in C^0 (\overline{\Omega})\,:\,v|_{S}\in\mathbb{P}_1\,\forall\, S\in T_{n,h}\}$.
\end{center}
We set $V_{n,h}:=X_{n,h}\cap H^1 (\Omega)$. We have that $V_{n,h}\subset V(\Omega,K_n)$, it is a finite dimensional space of dimension $N_h$, where $N_h$ is the number of inner nodes of $T_{n,h}$. The semi-discrete approximation problem is the following:\\
for $f\in C^{0,\alpha} ([0,T];L^2 (\Omega,m))$, $u_h^0\in V_{n,h}$ such that $u_h^0\to u_0$ in $L^2(\Omega,m)$ and for each $t\in [0,T]$, find $u_{n,h}(t)\in V_{n,h}$ such that
\begin{equation}\label{semidiscreto}
(\overline{P_{n,h}})
\begin{cases}
\displaystyle\left(\frac{d u_{n,h}}{dt}(t),v_h\right)_{L^2(\Omega,m)}+E(u_{n,h}(t),v_h)=(f,v_h)_{L^2(\Omega,m)}\quad\forall\,v_h\in V_{n,h}\\[4mm]
u_{n,h} (0)=u_h^0.
\end{cases}
\end{equation}
\noindent The existence and uniqueness of the semi-discrete solution $u_{n,h}(t)\in V_{n,h}$ of problem $(\overline{P_{n,h}})$ follows since problem $(\overline{P_{n,h}})$ is a Cauchy problem for a system of first order linear ordinary differential equations with constant coefficients (see e.g. \cite{quarval}). By proceeding as in the proof of Proposition \ref{stabilita}, we can prove an estimate similar to \eqref{stima2}, hence we have the stability of the method.\\
\noindent We now recall some key estimates of the interpolation error (see Proposition 4, Lemma 1 and Theorem 5.1 in \cite{lancef} and the references listed in).

\begin{teo} Let $u(t)$ be the solution of problem \eqref{pbforte2} and let $I_h (u)$ be the interpolant polynomial of $u$. Then, for every $t\in [0,T]$ it holds that
\begin{equation}\label{eqstima}
\|\nabla(u(t)-I_h (u(t)))\|^2_{L^2(\Omega)}+\|u(t)-I_h (u(t))\|^2_{H^1(K_n)}\leq c\,h^2\left(\|u\|^2_{H^2_{\mu} (\Omega)}+\|u\|^2_{H^2 (K_n)}\right),
\end{equation}
where $c$ is a positive constant.
\end{teo}

\begin{prop} Let $u(t)$ and $I_h(u(t))$ be as above. Then there exists a constant $C>0$ independent of the triangle $S$ such that
\begin{equation}\label{prop1}
\|u(t)-I_h(u(t))\|^2_{L^2(S)}\leq C\,h^4_S\frac{1}{\rho_S^{2\mu}}\,\sum_{|\alpha|=2}\int_S r(x)^{2\mu}\,|D^{\alpha} u(t)|^2\,dx,
\end{equation}
where $h_S$ is the diameter of the triangle $S\in T_{n,h}$ and $\rho_S$ is the radius of the biggest circle inscribed in $S$.
\end{prop}

\begin{prop} Let $u(t)$ and $I_h(u(t))$ be as above. Then for every $t\in [0,T]$ there exists a constant $C>0$ such that
\begin{equation}\label{prop2}
\|u(t)-I_h(u(t))\|^2_{L^2(\Omega)}\leq C\,h^4\|u(t)\|^2_{H^2_{\mu}}.
\end{equation}
\end{prop}

\noindent We now give an optimal error estimate with respect to the norm of $L^2([0,T];V(\Omega,K_n))$ for piecewise linear polynomials only.
\begin{teo}\label{stimaerrore} Let $u$ be the solution of \eqref{pbforte2} and $u_{n,h}$ be the discrete solution of \eqref{semidiscreto}. Then, for each $t\in [0,T]$, it holds that
\begin{center}\label{fordis}
$\displaystyle\|u(t)-u_{n,h}(t)\|_{L^2(\Omega,m)}^2+\int_0^t \|u(\tau)-u_{n,h}(\tau)\|_{V(\Omega,K_n)}^2\,d\tau\leq\|u_0-u_h^0\|^2_{L^2(\Omega,m)}+C\,h^2\int_0^t \|f(\tau)\|^2_{L^2(\Omega,m)}\,d\tau$,
\end{center}
where $C$ is a suitable constant independent of $h$.
\end{teo}

\noindent For the proof we refer to Theorem 5.2 in \cite{lancef} with small suitable changes.\\
We now consider the fully discretized problem, obtained by applying a finite difference scheme, the so-called $\theta$-method, on the time variable. It is well known that the $\theta$-method is unconditionally stable with respect both to the $L^2(\Omega)$ norm and to the $L^2(\Omega,m)$ norm provided $\frac{1}{2}\leq\theta\leq 1$. On the contrary, in the case of $0\leq\theta<\frac{1}{2}$, one has to assume that $\{T_{n,h}\}$ is a quasi-uniform family of triangulations and that a restriction on the time step holds. Since the peculiarity of our mesh $\{T_{n,h}\}$ is not to be quasi-uniform, from now on we assume $\frac{1}{2}\leq\theta\leq 1$. An error estimate between the semi-discrete solution $u_{n,h}(t)$ and the fully discrete one $u^l_{n,h}$ can be obtained as in Theorem 6.1 in \cite{lancef}. From this estimate and Theorem \ref{stimaerrore} we deduce the following convergence result.

\begin{teo} We set $t_l=l\Delta t$ for $l=0,1,\dots,\mathcal{M}$, $\Delta t>0$ being the time step and $\mathcal{M}$ being the integer part of $T/\Delta t$. Assume that $f\in C^{0,\alpha}([0,T]; L^2(\Omega,m))$ and $\frac{\partial f}{\partial t}\in L^2([0,T]\times\Omega, dt\times dm)$. Let $n$ be fixed and $u(t)$ be the solution of problem \eqref{pbforte2}, and let $u^l_{n,h}$ be the fully discretized solution with the same initial datum $u_h^0$ as given by the $\theta$-method with $\frac{1}{2}\leq\theta\leq 1$. Then for every $l=0,1,\dots,\mathcal{M}$
\begin{center}
$\displaystyle\|u(t_l)-u^l_{n,h}\|^2_{L^2(\Omega,m)}\leq \|u_0-u_h^0\|^2_{L^2(\Omega,m)}+C\,h^2 \left(\int_0^T \| f(\tau)\|^2_{L^2(\Omega,m)}\,d\tau\right)+C_\theta\,\Delta t^2\left(\left\|\frac{d u_{n,h}}{dt}(0)\right\|^2_{L^2(\Omega,m)}+\int_0^T \left\|\frac{\partial f}{\partial t} (\tau)\right\|^2_{L^2(\Omega,m)}\,d\tau\right)$,
\end{center}
where $C_\theta$ is a constant independent of $\mathcal{M}$, $\Delta t$ and $h$ and is a non-decreasing function of the continuity constant of $E(\cdot,\cdot)$ and $T$.
\end{teo}

\begin{oss} We point out that the norm $\left\|\frac{d u_{n,h}}{dt}(0)\right\|^2_{L^2(\Omega,m)}$ appearing in the above theorem can be estimated by $\|Au_0+f(0)\|^2_{L^2(\Omega,m)}$. Indeed, proceeding as in Remark 11.3.1 in \cite{quarval} with suitable changes, we take $u_h^0=\Pi^k_{1,h}(u_0)$, where $\Pi^k_{1,h}$ is the "elliptic projection operator". Hence, taking into account \eqref{corr}, we get
\begin{center}
$\displaystyle\left\|\frac{d u_{n,h}}{dt}(0)\right\|^2_{L^2(\Omega,m)}=\left(Au_0+f(0),\frac{d u_{n,h}}{dt}(0)\right)_{L^2(\Omega,m)}$,
\end{center}
and the thesis follows from Cauchy-Schwarz inequality.
\end{oss}

\section{Numerical results and conclusions}\label{sec6}
\setcounter{equation}{0}

\noindent In this section we present some numerical results concerning the transmission problem defined at the end of Section \ref{sec4}. We consider the domain illustrated in Figure \ref{fig:domain}. A highly conductive prefractal interface $K_n = K_{n,a} \cup K_{n,b}$, delimiting a non-convex polygonal domain $\Omega_1$, is placed at the center of a square domain $\Omega_2$ to study the heat transmission across the prefractal. In order to appreciate its role in the transmission process, we consider the nonlocal term $\theta_2(u)$ active only on the portion $K_{n,a}$ of the prefractal interface (in red in the figure). Defining symmetric conditions with respect to the geometry of the problem (in terms of boundary conditions and heat sources), we will be able to compare the heat flux that crosses the interface where the nonlocal term is present with the heat flux that crosses the interface where the nonlocal term is not active, and evaluate, numerically, the influence of the nonlocal term in the heat exchange process.

\begin{figure}
    \centerline{\includegraphics[scale=.5]{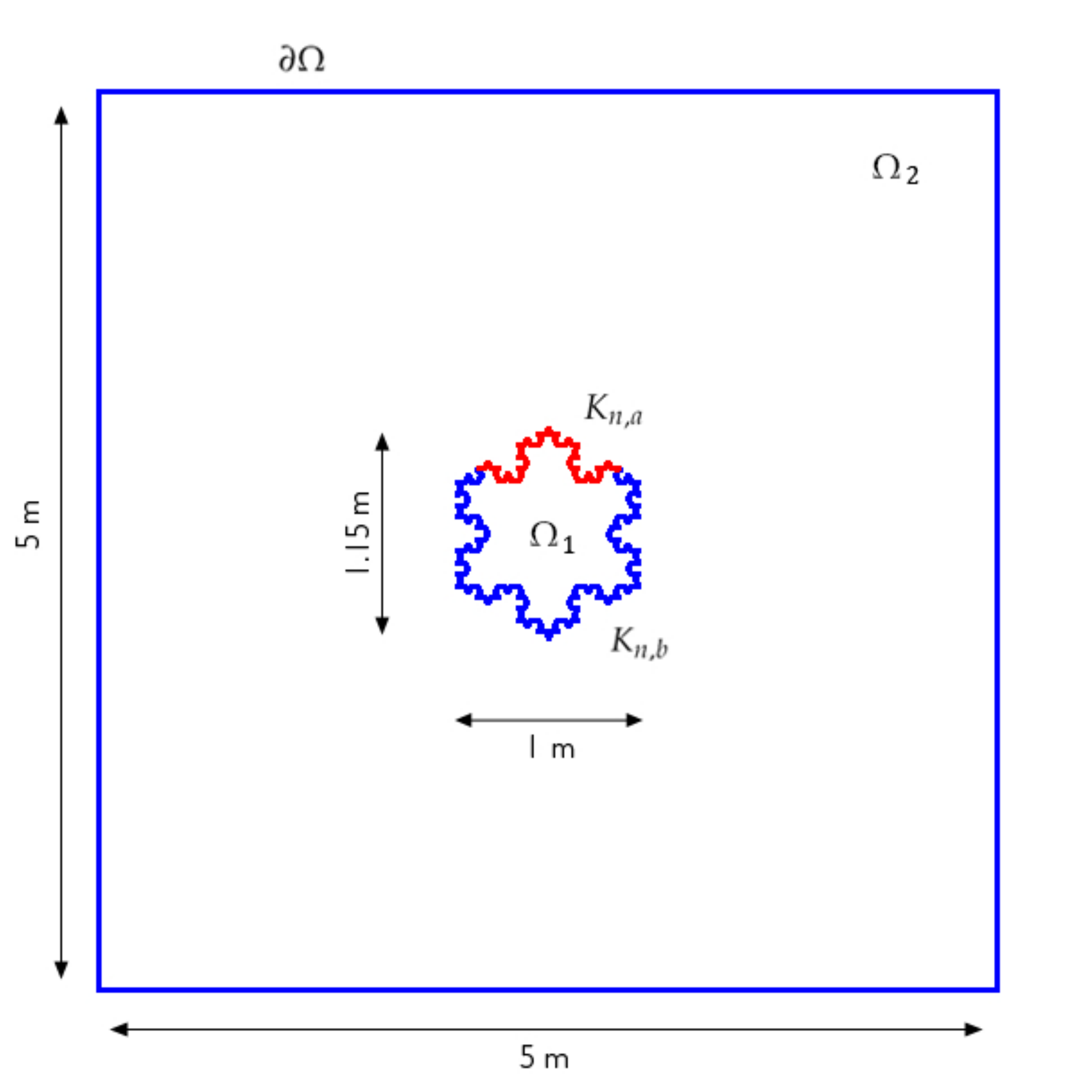}}
	\caption{The domain of the problem.}
	\label{fig:domain}
\end{figure}
\noindent

\noindent The dimensional equations of the problem are:
\begin{center}
$\begin{cases}
\rho \ C_p \ \frac{du}{dt}=k \ \Delta u + f & \text{in $L^2(\Omega)$,}\\[2mm]
\rho_s  \ C_{p,s} \ \frac{du}{dt}=-(\nu_1 \ k_1 \ \nabla u_1 + \nu_2 \ k_2 \ \nabla u_2) + k_s \ (\Delta_{K_{n,a}} u -\theta_2(u)) +f\quad &\text{in $L^2 (K_{n,a})$},\\[2mm]
\rho_s \ C_{p,s} \ \frac{du}{dt}=-(\nu_1 \ k_1 \ \nabla u_1 + \nu_2 \ k_2 \ \nabla u_2) + k_s \ \Delta_{K_{n,b}} u +f\quad &\text{in $L^2 (K_{n,b})$},\\[2mm]
u(0,x)=u_0(x) &\forall\,x \in\overline\Omega, \\[2mm]
u(t,x)=0 & \forall\,t, \forall\,x \in \partial \Omega
\end{cases}$
\end{center}
where
\begin{itemize}
\item $\Omega = \Omega_1 \cup \Omega_2$;
\item $\rho$ is the material density in the bulk domain $\Omega$ (in Kg/m$^3$);
\item $\rho_s$ is the material density per meter in the boundary domain $K_{n}$ (in Kg/m$^2$);
\item $C_p$ and $C_{p,s}$ are the heat capacity at constant pressure (in J/(Kg $\cdot$ K));
\item $k$ is the thermal conductivity in $\Omega$ (in W/(m $\cdot$ K)); $k_1 = k|_{\Omega_1}$ and $k_2 = k|_{\Omega_2}$;
\item $k_s$ is the thermal conductivity per meter in $K_n$ (in  W/K);
\item $\nu_1$ and $\nu_2$ are the outwards normal vectors on $K_n$ for $\Omega_1$ and $\Omega_2$ respectively;
\item the term $f$ represents a thermal source (in W/m$^3$);
\item $u$ is the unknown variable: the temperature in Kelvin degrees; \\$u_1 = u|_{\Omega_1}$ and $u_2 = u|_{\Omega_2}$.
\end{itemize}

\noindent The term $b u$ which appears in the equations of the problem defined in Section \ref{sec4} has been omitted here ($b=0$) to emphasize the role of the nonlocal term $\theta_2(u)$ in the transmission problem. The operator $\theta_2(u)$ is defined by the duality pairing between $H^{-\frac{1}{2}} (K_n)$ and $H^{\frac{1}{2}} (K_n)$. For every $u,v \in H^{\frac{1}{2}} (K_n)$, we define
\begin{center}
$\displaystyle\langle\theta_2 (u),v\rangle_{H^{-\frac{1}{2}} (K_n), H^{\frac{1}{2}} (K_n)}=\iint_{K_n\times K_n}\frac{(u(x)-u(y))(v(x)-v(y))}{|x-y|^2}\,ds(x)\,ds(y)$.
\end{center}

\noindent Observe that the term $\langle\theta_2 (u),v\rangle$ can be rewritten in the following way:
\begin{center}
$\displaystyle\langle\theta_2(u),v\rangle=\iint_{K_n\times K_n}\frac{(u(x)-u(y))v(x)}{|x-y|^2}\,ds(x)\,ds(y)-\iint_{K_n\times K_n}\frac{(u(x)-u(y))v(y)}{|x-y|^2}\,ds(x)\,ds(y)=\iint_{K_n\times K_n}\frac{(u(x)-u(y))v(x)}{|x-y|^2}\,ds(x)\,ds(y)+\iint_{K_n\times K_n}\frac{(u(y)-u(x))v(y)}{|x-y|^2}\,ds(x)\,ds(y)=2\iint_{K_n\times K_n}\frac{(u(x)-u(y))v(x)}{|x-y|^2}\,ds(x)\,ds(y)=2\int_{K_n} (Iu)(x)v(x)\,ds(x)$,
\end{center}
where $\displaystyle(Iu)(x):=\int_{K_n}\frac{(u(x)-u(y))}{|x-y|^2}\,ds(y)$. The last expression has been exploited for the implementation of the problem in the weak form. The simulations have been performed on Comsol V.3.5a, on a desktop computer with a quad-core Intel processor (i5-2320) running at 3.00 GHz
and equipped with 8 GB RAM.

\noindent Table \ref{table:physical_coeff} shows the numerical values used for the parameters above defined.

\begin{table}[h]
\begin{center}
\begin{tabular}{ c | c | c | c | c }
  $\rho$ & $\rho_s$ & $C_p$ & $C_{p,s}$ & $k_s$ \\
  \hline			
  8000 & 21000 & 450 & 150 & $10^6$
\end{tabular}
\caption{Numerical values used in the simulations for the physical coefficients.}
\label{table:physical_coeff}
\end{center}
\end{table}

\noindent Instead, the thermal conductivity $k$ has been defined variable in $\Omega$ as shown in Figure \ref{fig:domain_conductivity_def}. The domain has been ideally divided into eight sectors. The thermal conductivity is constant within each sector and variable from one sector to the subsequent one, so as to have alternations of very low values ($k = 1$: isolating material) and high values ($k = 1000$: good conductive material) between adjacent sectors.

\noindent The alternation of high and low values of the thermal conductivity in adjacent sectors separated by the prefractal layer is used to force the heat flow \emph{along} the prefractal on the east and west parts of the barrier, and \emph{through} the barrier on the north and south parts.

\begin{figure}
    \centerline{\includegraphics[scale=.35]{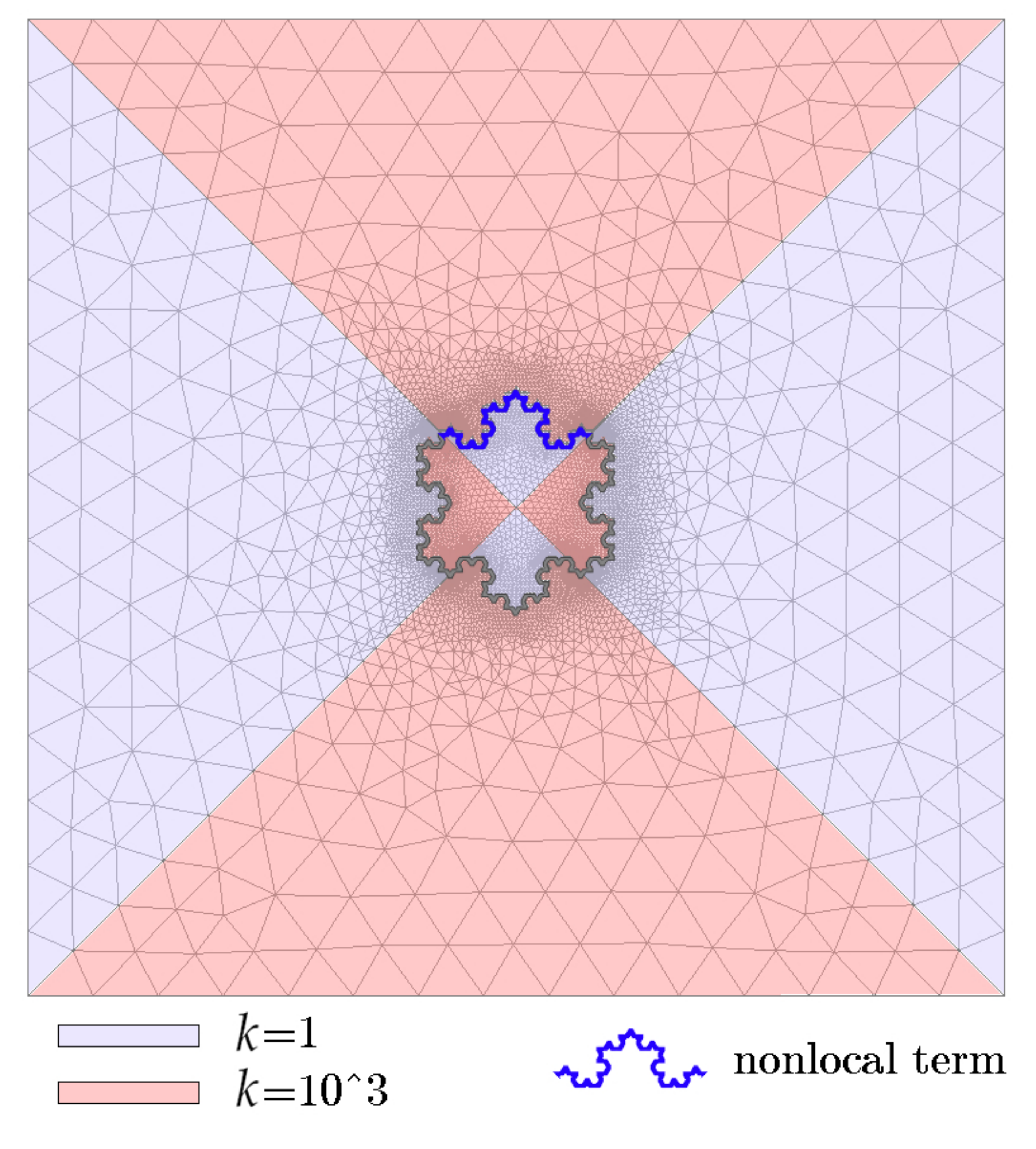}}
	\caption{Definition of the thermal conductivity in the bulk domain $\Omega$.}
	\label{fig:domain_conductivity_def}
\end{figure}
\noindent

\noindent The thermal source is defined as a 2D gaussian curve with a very low variance, in order to represent a flame concentrated at the center of $\Omega_1$:
\[ f(x) = 10^5 \ e^{-0.5 ((x_1-\bar{x}_1)^2+(x_2-\bar{x}_2)^2)/0.001}
\]
where $(\bar{x}_1, \bar{x}_2)$ are the coordinates of the center of $\Omega_1$. Figure \ref{fig:source_term} shows a 3D representation of the source term.

\begin{figure}
    \centerline{\includegraphics[scale=.9]{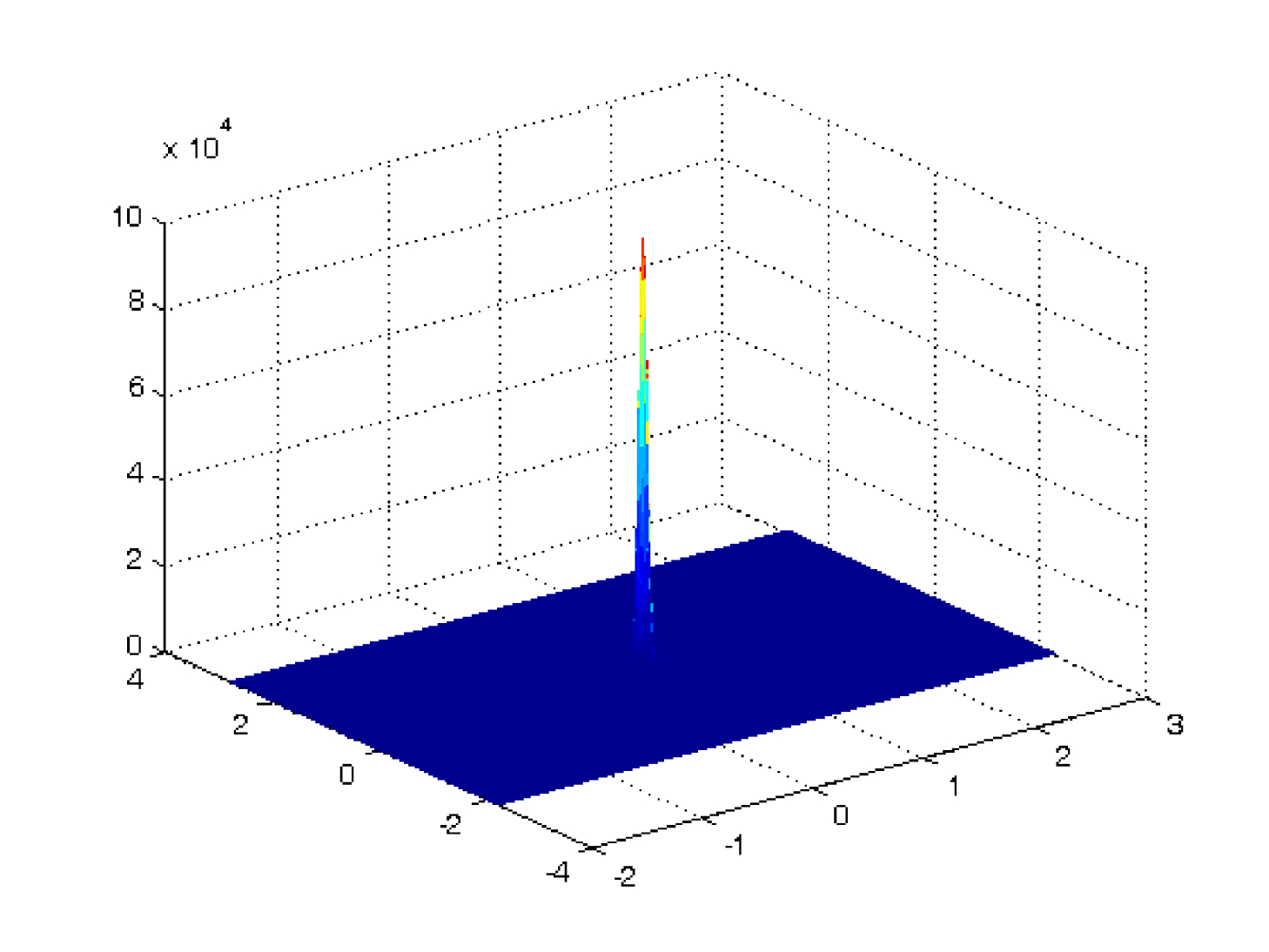}}
	\caption{The source term $f(x)$.}
	\label{fig:source_term}
\end{figure}
\noindent

\noindent Taking into account our choices on the location of the source term and the boundary conditions, the heat flows from the center of $\Omega_1$ (where the heat source has the maximum) towards the domain $\Omega_2$ and reaches the boundary $\partial \Omega$ where the temperature is kept constant at 0 (Dirichlet conditions). In the east and west sectors in $\Omega_1$, the heat produced by the source travels in the domain pushed by a high conductivity and reaches the prefractal barrier along the shortest possible path (ideally a straight line, which in the simulation takes the form of a slightly curved line because of numerical errors induced by the finite triangulation of the domain). As the barrier is reached, only a small part of the heat passes through it, because on the other side of the barrier there are two low conductivity areas that are holding the thermal flow.
The heat mainly flows \emph{along} the barrier (which is by assumption a highly conductive layer) until it reaches the north and south sectors, where, beyond the barrier there are again high conductivity areas.

\noindent Summarizing, the heat moves from the center of the domain $\Omega$ to the boundary $\partial \Omega$, and crosses the fractal layer mainly on the north and south sectors. The difference of the flow entity along these two main directions is due the role of the nonlocal term. The numerical simulations confirm that the nonlocal term is responsible for a larger flux across the barrier in the north sector.

\noindent Figure \ref{fig:streamlines} shows the main streamlines of the heat flux for the stationary solution.

\begin{figure}
    \centerline{\includegraphics[scale=.55]{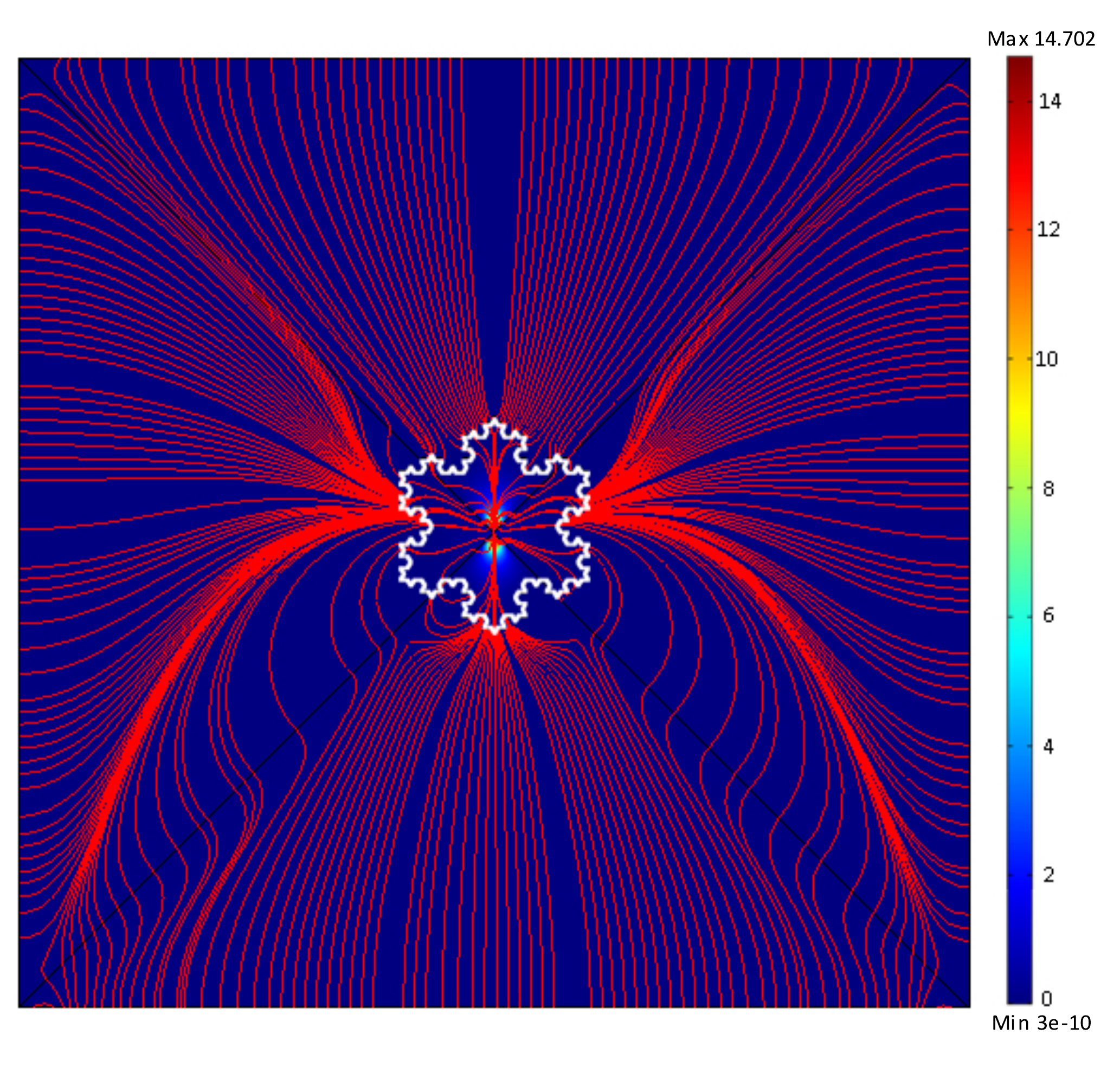}}
	\caption{Heat flux streamlines at the stationary condition.}
	\label{fig:streamlines}
\end{figure}
\noindent

\begin{figure}
    \centerline{\includegraphics[scale=.35]{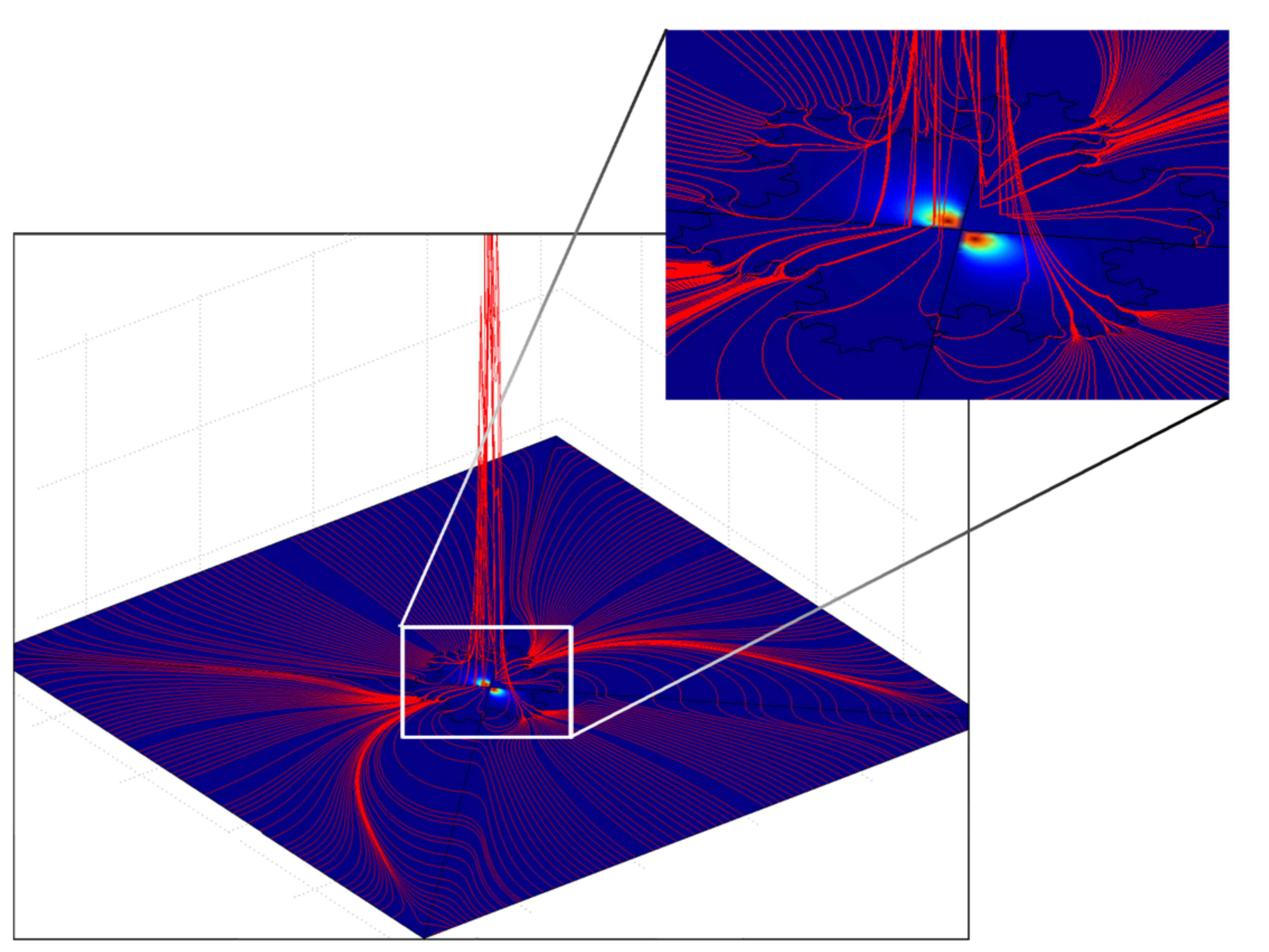}}
	\caption{A 3D representation of the heat flux streamlines. The height of the curves is proportional to the amplitude of the heat flux.}
	\label{fig:streamlines3D}
\end{figure}
\noindent

\noindent The streamlines have been drawn with a density in the domain proportional to the magnitude of the vector field to which they are tangent. Observe that the prefractal is crossed by much more lines in the north sector than everywhere else. This means that the amplitude of the heat flux that crosses the barrier in the north sector is much higher than in other sectors, and this is due to the presence of the nonlocal term.

\noindent Figure \ref{fig:streamlines3D} shows a three-dimensional representation of the same streamlines. The height of the curves is proportional to the magnitude of the heat flow. This figure confirms that the heat flux across the barrier in the east and west sectors is negligible (the corresponding curves are almost completely flat). Most of the flux across the barrier takes place in the north and south sectors. But the former is populated by much more lines, in virtue of the fact that the nonlocal term acting in the north part of the prefractal is responsible of a larger heat flux across the barrier.

\noindent We conclude by noticing that the same results may be obtained also defining the problem on different domains. The Koch curve could be replaced by a more general symmetric prefractal of any order, or even by a prefractal of mixture type \cite{haodong}. As already pointed out in \cite{lancef}, the fractal geometry helps to achieve a larger heat flux across the barrier. Our experimental results suggest that by drawing a prefractal barrier of a proper material characterized by non-constant heat conductivity (which may be described by the nonlocal term $\theta_2(u)$) one could obtain a highly conductive layer with increased capability to drain the heat.

\bigskip

\noindent {\bf Acknowledgements.} The authors have been supported by the Gruppo
Nazionale per l'Analisi Matematica, la Probabilit\`a e le loro
Applicazioni (GNAMPA) of the Istituto Nazionale di Alta Matematica
(INdAM).


\medskip

\begin{thebibliography}{3}

\bibitem{adhei} D. R. Adams, L. I. Hedberg, \emph{Function Spaces and Potential Theory}, Springer-Verlag, 1966.

\bibitem{Ar-Me-Pa-Ro} W. Arendt, G. Metafune, D. Pallara, S. Romanelli, {\em The Laplacian with Wentzell-Robin Boundary Conditions on Spaces of Continuous Functions}, Semigroup Forum,  67, (2003), 247-261.


\bibitem{baiocap} C. Baiocchi, A. Capelo, \emph{Variational and Quasivariational Inequalities: Applications to Free-Boundary Value Problems}, Wiley, New York, 1984.

\bibitem{bass-Levin} R. Bass, D. Levin, {\em Transition probabilities for symmetric jump processes}, Trans. Amer. Math. Soc., 354, (2002), 7, 2933-2953.


\bibitem{brezis} H. Brezis, \emph{Analisi funzionale}, Liguori, Napoli, 1986.

\bibitem{bregil} F. Brezzi, G. Gilardi, \emph{Fundamentals of PDE for numerical analysis}, in: Finite Element Handbook, McGraw-Hill Book Co., New York, 1987.

\bibitem{canmey} J. R. Cannon, G. H. Meyer, {\em On a diffusion in a fractured medium}, SIAM J. Appl. Math., 3, (1971), 434-448 .

\bibitem{cefalolancia} M. Cefalo, M. R. Lancia, \emph{An optimal mesh generation algorithm for domains with Koch type boundaries}, Math. Comput. Simulation, 106, (2014), 136-162.

\bibitem{haodong} M. Cefalo, M. R. Lancia, H. Liang, \emph{Heat-flow problems across fractals mixtures: regularity results of the solutions and numerical approximations}, Differential and Integral Equations, Vol. 26, Numbers 9-10, (2013), 1027-1054.

\bibitem{crelannazver} S. Creo, M. R. Lancia, A. Nazarov, P. Vernole, \emph{On two-dimensional nonlocal Venttsel' problems in piecewise smooth domains}, arXiv:1702.06324.

\bibitem{nostro} S. Creo, M. R. Lancia, A. V\'elez-Santiago, P. Vernole, {\em Approximation of a nonlinear fractal energy functional on varying Hilbert spaces}, submitted, 2016.

\bibitem{evans} L. C. Evans, \emph{Partial Differential Equations}, American Mathematical Society, 1998.

\bibitem{falconer} K. Falconer, \emph{The Geometry of Fractal Sets, 2nd ed.}, Cambridge University Press, 1990.

\bibitem{farkas} W. Farkas, N. Jacob, \emph{Sobolev spaces on non-smooth domains and Dirichlet forms related to subordinate reflecting diffusions}, Math. Nachr. 224, (2001), 75-104.

\bibitem{Fa-G-G-Ro} A. Favini, G. R. Goldstein, J. A. Goldstein, S. Romanelli, {\em The heat equation with generalized Wentzell boundary condition}, J. Evol. Equ., 2, (2002), 1, 1-19.

\bibitem{Fa-La-Le-Ma} A. Favini, R. Labbas,  K. Lemrabet, S. Maingot, {\em Study of the limit of transmission problems in a thin layer by the sum theory of linear operators}, Rev. Mat. Complut., 18, (2005), 143-176.

\bibitem{freiberg} U. R. Freiberg, M. R. Lancia, \emph{Energy form on a closed fractal curve}, Z. Anal. Anwendingen, 23, 1, (2004), 115-135.

\bibitem{fukush} M. Fukushima, Y. Oshima, M. Takeda, \emph{Dirichlet Forms and Symmetric Markov Processes}, de Gruyter Studies in Mathematics, Vol. 19, W. de Gruyter, Berlin, 1994.

\bibitem{grisvard} P. Grisvard, \emph{Elliptic Problems in Nonsmooth Domains}, Pitman, Boston, 1985.


\bibitem{hardy} G. Hardy, J. Littlewood, G. Polya, \emph{Inequalities}, Cambridge University Press, Cambridge, 1952.

\bibitem{jerison} D. Jerison, C. E. Kenig, \emph{The inhomogeneous Dirichlet Problem in Lipschitz domains}, Journal of Functional Analysis, 130, (1995), 161-219.

\bibitem{kenig} D. Jerison, C. E. Kenig, \emph{The Neumann problem on Lipschitz domains}, Bull. Amer. Math. Soc. (N.S.), 4, (1981), 203-207.

\bibitem{kato} T. Kato, \emph{Perturbation theory for linear operators}, II edit., Springer, 1977.

\bibitem{kond} V. A. Kondrat'ev, \emph{Boundary-value problems for elliptic equations in domains with conical or angular point}, Trans. Moscow Math. Soc., 16, (1967), 209-292.

\bibitem{Kor} P. Korman, {\em Existence of periodic solutions for a class of nonlinear problems}, Nonlinear Anal., 7, (1983), 873-879.

\bibitem{lancef} M. R. Lancia, M. Cefalo, G. Dell'Acqua, \emph{Numerical approximation of transmission problems across Koch-type highly conductive layers}, Applied Mathematics and Computation, 218, (2012), 5453-5473.

\bibitem{lanmos} M. R. Lancia, U. Mosco, M. A. Vivaldi, \emph{Homogenization for conductive thin layers of pre-fractal type}, J. Math. Anal. Appl., 347, (2008), 354-369.

\bibitem{LVSV} M. R. Lancia, A. V\'elez-Santiago, P. Vernole, {\em Quasi-linear Venttsel' problems with nonlocal boundary conditions}, Nonlinear Anal. Real World Appl., 35, (2017), 265-291.

\bibitem{lanver1} M. R. Lancia, P. Vernole, \emph{Convergence results for parabolic transmission problems across highly conductive layers with small capacity}, Adv. Math. Sc. Appl., 16, (2006), 411-445.

\bibitem{lanver2} M. R. Lancia, P. Vernole, \emph{Venttsel' problems in fractal domains}, Journal of Evolution Equations, Vol. 14, Issue 3, (2014) 681-712.



\bibitem{lunardi} A. Lunardi, \emph{Analytic semigroups and optimal regularity in parabolic problems}, Progress in Nonlinear Differential Equations and their Applications, 16, Birkh{\"a}user Verlag, Basel, 1995.

\bibitem{mazya} V. G. Maz'ya, \emph{Sobolev Spaces with Applications to Elliptic Partial Differential Equations}, Springer-Verlag, 2011.

\bibitem{nazplam} S. A. Nazarov, B. A. Plamenevsky, \emph{Elliptic Problems in Domains with Piecewise Smooth Boundaries}, de Gruyter, Berlin-New York, 1994.

\bibitem{necas} J. Necas, \emph{Les methodes directes en theorie des equationes elliptiques}, Masson, 1967.

\bibitem{pazy} A. Pazy, \emph{Semigroup of linear operators and applications to partial differential equations}, Applied Mathematical Sciences, 44, Springer-Verlag, New York, 1983.

\bibitem{sanpal} H. Pham Huy, E. Sanchez-Palencia, {\em Ph\`enom\'enes des transmission \'a travers des couches minces de conductivit\`e \`elev\`ee}, J. Math. Anal. Appl., 47, (1974), 284-309.

\bibitem{quarval} A. Quarteroni, A. Valli, \emph{Numerical Approximation of Partial Differential Equations}, Springer-Verlag, 1997.

\bibitem{Shim} M. Shinbrot, {\em Water waves over periodic bottoms in three dimensions}, J. Inst. Math. Appl., 25, (1980), 4, 367-385.

\bibitem{showal} R. E. Showalter, \emph{Hilbert space methods for partial differential equations}, Monographs and studies in Mathematics, Pitman, 1977.

\bibitem{velezjfa14} A. V{\'e}lez-Santiago, \emph{Quasi-linear variable exponent boundary value problems with {W}entzell-{R}obin and {W}entzell boundary conditions}, J. Functional Analysis, 266, (2014), 560-615.

\bibitem{velezjfa15} A. V{\'e}lez-Santiago, \emph{Global regularity for a class of quasi-linear local and nonlocal elliptic equations on extension domains}, J. Functional Analysis, 269, (2015), 1-46.

\bibitem{AVS-WAR} A. V{\'e}lez-Santiago, M. Warma, \emph{A class of quasi-linear parabolic and elliptic equations with nonlocal {R}obin boundary conditions}, J. Math. Anal. Appl., 372, (2010), 120-139.

\bibitem{vent59} A. D. Venttsel', {\em On boundary conditions for multidimensional diffusion processes,} Teor. Veroyatnost. i Primenen., 4, (1959), 172-185; English translation, Theor. Probability Appl., 4, (1959), 164-177.


\bibitem{WAR12-2} M. Warma, \emph{Regularity and well-posedness of some quasi-linear elliptic and parabolic problems with nonlinear general {W}entzell boundary conditions on nonsmooth domains}, Nonlinear Analysis, 14, (2012), 5561-5588.

\bibitem{WAR12} M. Warma, \emph{The $p$-Laplace operator with the nonlocal Robin boundary conditions on arbitrary open sets}, Annali Math. Pura Appl., 193, (2014), 771-800.


\end{thebibliography}
\end{document}